\documentclass[a4paper,11pt,french,english]{amsart}
\usepackage[english]{babel}
\usepackage{inputenc}
\usepackage{lmodern}
\usepackage[T1]{fontenc}
\usepackage{amsmath,amsthm,amssymb,amsfonts}
\usepackage{amscd}
\usepackage{pstricks}
\usepackage[a4paper,left=4cm,right=4cm,top=4cm,bottom=4cm]{geometry}
\usepackage{enumerate}
\usepackage{float}
\newcommand*{\tree}{\mathcal{T}_{d,k}}
\newcommand*{\treed}{\mathcal{T}_d}
\newcommand*{\bord}{\partial_{\infty} \mathcal{T}_{d,k}}
\newcommand*{\bordd}{\partial_{\infty} \mathcal{T}_d}
\newcommand*{\aut}{\mathrm{Aut}(\mathcal{T}_{d,k})}
\newcommand*{\autd}{\mathrm{Aut}(\mathcal{T}_{d})}
\newcommand*{\aaut}{\mathrm{AAut}(\mathcal{T}_{d,k})}
\newcommand*{\aautd}{\mathrm{AAut}(\mathcal{T}_d)}
\newcommand*{\Daaut}{\mathrm{AAut}_{D}(\mathcal{T}_{d,k})}
\newcommand*{\Gaaut}{\mathrm{AAut}_G(\mathcal{T}_d)}
\newcommand*{\Haaut}{\mathrm{AAut}_H(\mathcal{T}_d)}
\newcommand*{\rpc}{/ \! \! /}

\newcommand*{\sch}{Schlichting }

\title{Compact presentability of tree almost automorphism groups}
\author{Adrien Le Boudec}
\address{Laboratoire de Math\'ematiques, b\^atiment 425, Universit\'e Paris-Sud 11, 91405 Orsay, France}
\email{adrien.le-boudec@math.u-psud.fr}
\date{June 16, 2016}

\theoremstyle{plain}
\newtheorem{thm}{Theorem}[section]

\newtheorem{prop}[thm]{Proposition}
\newtheorem{cor}[thm]{Corollary}

\newtheorem{lem}[thm]{Lemma}

\theoremstyle{definition}
\newtheorem{defi}[thm]{Definition}
\newtheorem{ex}[thm]{Example}
\newtheorem{rmq}[thm]{Remark}

\begin{document}

\maketitle

\begin{abstract}
We establish compact presentability, i.e.\ the locally compact version of finite presentability, for an infinite family of tree almost automorphism groups. Examples covered by our results include Neretin's group of spheromorphisms, as well as the topologically simple group containing the profinite completion of the Grigorchuk group constructed by Barnea, Ershov and Weigel.

We additionally obtain an upper bound on the Dehn function of these groups in terms of the Dehn function of an embedded Higman-Thompson group. This, combined with a result of Guba, implies that the Dehn function of the Neretin group of the regular trivalent tree is polynomially bounded. 
\end{abstract}

\setcounter{tocdepth}{1}
\tableofcontents

\section{Introduction}

\subsection*{Almost automorphism groups}

If $T$ is a locally finite tree, then its automorphism group $\mathrm{Aut}(T)$, endowed with its natural locally compact and totally disconnected topology, acts continuously and properly on $T$.

%More flexible than the notion of automorphism is the notion of almost automorphism.
Almost automorphisms of the tree $T$ (sometimes called spheromorphisms) do not act on $T$, but on its boundary $\partial_{\infty} T$. Roughly speaking, an almost automorphism of $T$ is a transformation induced in the boundary by a piecewise tree automorphism. Almost automorphisms form a topological group $\mathrm{AAut}(T)$ containing the automorphism group $\mathrm{Aut}(T)$ as an open subgroup.

In the case where $T$ is a non-rooted regular tree of degree $d+1 \geq 3$, the group $\mathcal{N}_d$ of almost automorphisms of $T$ was introduced by Neretin in connection with his work in representation theory \cite{Neretin}. Neretin proved that, from the point of view of representation theory, the group $\mathcal{N}_d$ can be seen as a $p$-adic analogue of the diffeomorphism group of the circle. Inspired by a simplicity result for the diffeomorphism group of the circle $\mathrm{Diff}^+(\mathbb{S}^1)$ \cite{Herm}, Kapoudjian later proved that the group $\mathcal{N}_d$ is abstractly simple \cite{simple}.

Recently, Bader, Caprace, Gelander and Mozes proved that $\mathcal{N}_d$ does not have any lattice \cite{BCGM}. This result is remarkable for the reason that all the familiar examples of simple locally compact groups (which are unimodular), e.g.\ real or $p$-adic Lie-groups, or the group of type preserving automorphisms of a locally finite regular tree, are known to have lattices. Actually $\mathcal{N}_d$ turned out to be the first example of a locally compact simple group without lattices.

In this paper we investigate a family of groups which appear as generalizations of Neretin's group. Here we give an outline of their construction (see Section \ref{aaut-type} for precise definitions). Every finite permutation subgroup $D \leq \mathrm{Sym}(d)$ is known to give rise to a closed subgroup of the automorphism group $\mathrm{Aut}(\mathcal{T}_d)$ of the rooted $d$-regular tree $\mathcal{T}_d$, by considering the infinitely iterated permutational wreath product $W(D) = (\ldots \wr D) \wr D$. Elements of $W(D)$ are rooted automorphisms whose local action is prescribed by $D$. We now consider the quasi-regular rooted tree $\tree$, in which the root has degree $k$ and other vertices have degree $d+1$. Roughly, the almost automorphism groups $\Daaut$ we are interested in, are homeomorphisms of $\bord$ that are piecewise tree automorphisms whose local action is prescribed by $D$. This family of groups generalizes Neretin's groups because when $D$ is the full permutation group $\mathrm{Sym}(d)$, the group $W(D)$ is the full automorphism group of $\mathcal{T}_d$ and we can check that $\mathrm{AAut}_{D}(\mathcal{T}_{d,2}) \simeq \mathcal{N}_d$.

These groups appear in \cite{germs}, where a careful study of the abstract commensurator group of self-replicating profinite wreath branch groups is carried out (we refer to \cite{BEW} for an introduction to abstract commensurators of profinite groups). Let $W_k(D)$ be the closed subgroup of $\aut$ fixing pointwise the first level of $\tree$ and acting by an element of $W(D)$ in each subtree rooted at level one. Under the additional assumption that $D \leq \mathrm{Sym}(d)$ is transitive and is equal to its normaliser in $\mathrm{Sym}(d)$, the group $\Daaut$ turns out to be isomorphic to the abstract commensurator group of $W_k(D)$. In particular Neretin's group $\mathcal{N}_d$ is the abstract commensurator group of $W_2(\mathrm{Sym}(d))$, or equivalently the group of germs of automorphisms of $\mathrm{Aut}(T)$ in the language of \cite{germs}.

The result that the groups $\Daaut$ are compactly presented is motivated by their connections with Thompson groups and their generalizations. Recall that Higman \cite{Higman} constructed an infinite family of finitely presented infinite simple groups $V_{d,k}$ (sometimes denoted $G_{d,k}$), generalizing the group $V$ introduced by R. Thompson. When $D \leq \mathrm{Sym}(d)$ is the trivial group then $\Daaut$ is nothing else than $V_{d,k}$ (see Section \ref{sec-Hig} for details). One of the reasons why combinatorial group theorists became interested in Thompson groups is because of the combination of simplicity and finiteness properties. Indeed Thompson groups $T$ and $V$ turned out to be the first known examples of finitely presented infinite simple groups (see \cite{CFP}). While simplicity results for $\Daaut$ have recently been obtained in \cite{germs}, here we settle in the positive the question whether these groups satisfy the locally compact version of being finitely presented, i.e.\ being compactly presented.

\subsection*{Compact presentability}

Recall that a locally compact group is said to be compactly generated if there exists a compact subset $S$ so that the group generated by $S$ is the whole group $G$. Less known than the notion of compact generation is the notion of compact presentation. A locally compact group $G$ is said to be compactly presented if it admits a compact generating subset $S$ such that $G$ has a presentation, as an abstract group, with $S$ as set of generators and relators of bounded length (but possibly infinitely many relators). When the group $G$ is discrete, this amounts to saying that $G$ is finitely presented; and as in the discrete case, for a locally compact group, being compactly presented does not depend on the choice of the compact generating set $S$.

Compact presentability can be interpreted in terms of coarse simple connectedness of the Cayley graph of the group with respect to some compact generating subset. In particular, among compactly generated locally compact groups, being compactly presented is preserved by quasi-isometries. For a proof of this result, see for instance \cite[Chapter 8]{Cor-dlH}.

Our first result is the following:

\begin{thm} \label{thma}
For any $k \geq 1, d \geq 2$, and any subgroup $D \leq \mathrm{Sym}(d)$, the group $\Daaut$ is compactly presented.
\end{thm}

The group $\Daaut$ always contain a dense copy of the Higman-Thompson finitely presented group $V_{d,k}$. Here we insist on the fact that for a locally compact group, although having a dense finitely generated subgroup is a sufficient condition for being compactly generated, the analogue result does not hold for compact presentation, i.e.\ having a dense finitely presented subgroup does not imply compact presentation of the ambient group. For example, for any non-Archimedean local field $\mathbb{K}$, the group \mbox{$\mathbb{K}^2 \rtimes \mathrm{SL}_2(\mathbb{K})$} has a central extension with non-compactly generated kernel, and is therefore not compactly presented (see for instance \cite[Proposition 8.A.26]{Cor-dlH}). However the reader can check that this group admits dense finitely generated free subgroups.

We also emphasize the fact that, for the case of Neretin's group, Theorem \ref{thma} cannot be obtained by proving finite presentation of a discrete cocompact subgroup because these do not exist \cite{BCGM}. However we note that it seems to be unknown whether Neretin's group $\mathcal{N}_d$ is quasi-isometric to a finitely generated group.

As a by-product of Theorem \ref{thma} and the main result of \cite{BCGM}, we also obtain that locally compact simple groups without lattices also exist in the realm of compactly presented groups.

\subsection*{Dehn function}

Having obtained compact presentability of a locally compact group $G$ naturally leads to the study of an invariant of $G$, having both geometric and combinatorial flavors, called the Dehn function of $G$.

From the geometric point of view, the Dehn function $\delta_G(n)$ is the supremum of areas of loops in $G$ of length at most $n$. In other words, it is the best isoperimetric function, where isoperimetric function can be understood as for simply connected Riemannian manifolds. 

From the combinatorial perspective, the Dehn function is a quantified version of compact presentability: $\delta_G(n)$ is the supremum over all relations $w$ of length at most $n$ in the group, of the minimal number of relators needed to convert $w$ to the trivial word.

First recall that for any two functions $f,g: \mathbb{N} \rightarrow \mathbb{N}$, we say that $f$ is asymptotically bounded by $g$, denoted $f \preccurlyeq g$, if for some constant $c$ we have $f(n) \leq cg(cn)+cn+c$ for every $n \geq 0$; and $f,g$ have the same $\approx$-asymptotic behavior, denoted by $f \approx g$, if $f \preccurlyeq g$ and $g \preccurlyeq f$. 

If $G$ is compactly presented and if $S$ is a compact generating set, then there is $k \geq 1$ such that the group $G$ is presented by $\left\langle S \mid R_k \right\rangle$, where $R_k$ is set of relations in $G$ of length at most $k$. The area $a(w)$ of a relation $w$, i.e.\ a word in the letters of $S$ which represents the identity in $G$, is the smallest integer $m$ so that $w$ can be written in the free group $F_S$ as a product of $m$ conjugates of relators of $R_k$. Now define the Dehn function of $G$ by \[ \delta_{G}(n) = \sup \left\{ a(w) : w \, \, \text{relation of length at most $n$}\right\}. \] This function depends on the choice of $S$ and $k$, but its $\approx$-asymptotic behavior does not, and is actually a quasi-isometry invariant of $G$.

Our second result is the following upper bound on the Dehn function of almost automorphism groups:

\begin{thm} \label{thmb}
For any $k \geq 1, d \geq 2$, and any subgroup $D \leq \mathrm{Sym}(d)$, the Dehn function of $\Daaut$ is asymptotically bounded by that of $V_{d,k}$.
\end{thm}

On the other hand, the Dehn function of $\Daaut$ is not linear because having a linear Dehn function characterizes Gromov-hyperbolic groups among compactly presented groups, and the group $\Daaut$ is easily seen not to be Gromov-hyperbolic. So by a general argument (see for example \cite{Bow}), the Dehn function of $\Daaut$ has a quadratic lower bound.

In the case $d=2$, all the groups $V_{2,k}$ turn out to be isomorphic to Thompson group $V$. While the Dehn function of Thompson group $F$ has been proved to be quadratic \cite{Guba-quad}, it is not known whether the Dehn function of $V$ is quadratic or not. However, using a result of Guba \cite{Guba-poly} who showed the upper bound $\delta_V \preccurlyeq n^{11}$, we obtain:

\begin{cor}
Neretin's group $\mathcal{N}_2$ has a polynomially bounded Dehn function ($\preccurlyeq n^{11}$).
\end{cor}

We believe that the result of Guba can be extended to the family of groups $V_{d,k}$, i.e.\ that every group $V_{d,k}$ satisfies a polynomial isoperimetric inequality. By Theorem \ref{thmb} this would imply that the Dehn function of $\Daaut$ is polynomially bounded for arbitrary $k \geq 1 , d \geq 2$ and $D \leq \mathrm{Sym}(d)$.

\subsection*{Almost automorphism groups associated with closed regular branch groups}

The notions of self-similarity and branching appear naturally in the theory of groups acting on rooted trees. Basic definitions are recalled at the beginning of Section \ref{Gaaut-type}, and we refer the reader to the surveys \cite{Nek-ssg,branch} for more on self-similar and branch groups.

To any self-similar group $G \leq \autd$ we naturally associate a subgroup $\Gaaut \leq \aautd$, consisting of almost automorphisms acting locally like an element of $G$. The group $\Gaaut$ always contains the Higman-Thompson group $V_d$ and is generated by $V_d$ together with an embedded copy of $G$. See Section \ref{Gaaut-type} for more details. It is worth noting that the definition of the group $\Gaaut$ makes sense when $G \leq \autd$ is an abstract subgroup. In particular we make a priori neither topological (e.g.\ closed) nor finiteness (e.g.\ finitely generated) assumption on $G$.

The first example of such a group was studied by R\"{o}ver in the case when $G$ is the first Grigorchuk group. He proved that $\Gaaut$ is finitely presented and simple \cite{Rov-cfpsg}. The case of a general self-similar group was then studied by Nekrashevych, who proved that these groups enjoy properties rather similar to the properties of the Higman-Thompson groups (see \cite{Nek-cp,Nek-fp}).

Later Barnea, Ershov and Weigel \cite{BEW} made use of R\"{o}ver's simplicity result to prove that the profinite completion of the Grigorchuk group, which coincides with its topological closure in $\mathrm{Aut}(\mathcal{T}_2)$, embeds as an open subgroup in a topologically simple group, namely the group of almost automorphisms acting locally like an element of the closure of the Grigorchuk group.

Here we are interested in almost automorphism groups associated with closed regular branch groups. These can also be seen as generalizations of Neretin's group. It turns out that in this setting, $\Gaaut$ is naturally a totally disconnected locally compact (t.d.l.c.\ for short) group, admitting $G$ as a compact open subgroup. Under some more assumptions on $G$, we prove: 

\begin{thm} \label{thm-cp-gaaut}
Let $G \leq \autd$ be the closure of some finitely generated, contracting regular branch group, branching over a congruence subgroup. Then $\Gaaut$ is a t.d.l.c.\ compactly presented group.
\end{thm}

Examples of groups covered by Theorem \ref{thm-cp-gaaut} include the aforementioned topologically simple group constructed in \cite{BEW}, as well as other groups described in Section \ref{Gaaut-type}. As an application, we obtain that the profinite completion of the Grigorchuk group embeds as an open subgroup in a topologically simple compactly presented group.

Note that any group $G$ appearing in Theorem \ref{thm-cp-gaaut} can be explicitly described in terms of the notions of patterns and finitely constrained groups, an introduction of which can be found in \cite{ZS-haus}: $G$ is the finitely constrained group defined by allowing all patterns of a fixed size appearing in the group of which it is the closure. See also the end of Section \ref{Gaaut-type}.

The proof of Theorem \ref{thm-cp-gaaut} will consist in two steps, and we believe that each one is of independent interest. The first will be to establish a general result on compact presentability of the so-called \sch completions (see Section \ref{sec-sch}), and the second will be to identify our group with the \sch completion of one of its dense subgroup (see Theorem \ref{aaut-sch}).

\subsection*{Organization}
We start by providing a brief introduction to almost automorphisms of trees and Higman-Thompson groups in the next two sections. In Section \ref{aaut-type} we define the groups $\Daaut$ and their topology, and establish some preliminary results. Section \ref{sec-pres} contains the proofs of Theorem \ref{thma} and Theorem \ref{thmb}. In Section \ref{sec-sch}, we prove a general statement on compact presentability of \sch completions, namely Theorem \ref{thm-sch}, and the proof of Theorem \ref{thm-cp-gaaut} is given in Section \ref{Gaaut-type}.

\subsection*{Acknowledgments}
I thank Yves de Cornulier for stimulating interactions, for his careful reading of the paper and his valuable comments. I am also very grateful to Pierre-Emmanuel Caprace for helpful discussions which gave rise to this work, to Laurent Bartholdi for useful discussions about branch groups, and to Pierre de la Harpe for his remarks and corrections.

\section{Tree almost automorphisms}

\subsection{The quasi-regular rooted tree $\tree$ and its boundary}
Let $A$ and $B$ be finite sets of cardinality respectively $k \geq 1$ and $d \geq 2$. Consider the set of finite words $\left\{\varnothing \right\} \cup \left\{a b_1\cdots b_n \, : \, a \in A, \, b_i \in B \right\}$ over the alphabet $X = A \cup B$ being either empty or beginning by an element of $A$. This set is naturally the vertex set of a rooted tree, where the root is the empty word $\varnothing$ and two vertices are adjacent if they are of the form $v$ and $vx$, $x \in X$. We will denote this tree by $\tree$. In the case when $k=d$ it will be denoted by $\mathcal{T}_{d}$. For any vertex $v$, we will also denote by $\tree^v$ the subtree of $\tree$ spanned by vertices having $v$ as a prefix. The distance between a vertex and the root will be called its level, and the number of its neighbours will be called its degree. If $v$ is a vertex of level $n \geq 0$, then its neighbours of level $n+1$ are called the descendants of $v$. By construction, the root of $\mathcal{T}_{d,k}$ has degree $k$, and a vertex of level $n \geq 1$ has degree $d+1$: it has one distinguished neighbour pointing toward the root, and $d$ descendants. See Figure \ref{figure1} for the case $k=2$, $d=3$.

The boundary $\bord$ of the tree $\mathcal{T}_{d,k}$ is defined as the set of infinite words $a b_1 \cdots b_n \cdots$, i.e.\ infinite geodesic rays in $\tree$ started at the root. We define the distance between two such words $\xi, \xi'$ by $d(\xi, \xi') = d^{-|\xi \wedge \xi'|}$, where $|\xi \wedge \xi'|$ is the length of the longest common prefix of $\xi$ and $\xi'$. Equipped with this distance, the boundary at infinity $\partial_{\infty} \tree$ turns out to be homeomorphic to the Cantor set.

From now and for the rest of the paper, we fix an embedding of $\tree$ in the oriented plane. This embedding induces a canonical way of ordering, say from left to right, the descendants of any vertex. In particular we obtain a total ordering on the boundary at infinity $\bord$, defined by declaring that $\xi \leq \xi'$ if the first letter of $\xi$ following the longest common prefix of $\xi$ and $\xi'$, is smaller than the one of $\xi'$. 

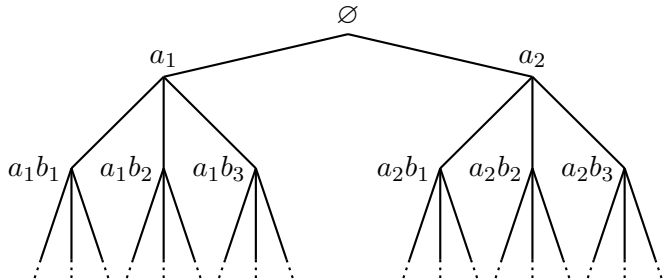
\begin{figure}[H]
\psset{unit=23pt}
\begin{pspicture}(13,5)
\psline[arrowsize=0.16]{-}(4,4)(7,4.7)
\psline[arrowsize=0.16]{-}(7,4.7)(10,4)
\rput(7,5){$\varnothing$}
\rput(4,4.3){$a_1$}
\rput(10,4.3){$a_2$}
%\rput(0,1){$2$}
%\rput(0,2.5){$1$}
%\rput(0,4){$0$}
\rput(1.9,2.5){$a_1b_1$}
\rput(3.4,2.5){$a_1b_2$}
\rput(4.9,2.5){$a_1b_3$}
\rput(7.9,2.5){$a_2b_1$}
\rput(9.4,2.5){$a_2b_2$}
\rput(10.9,2.5){$a_2b_3$}

\psline[arrowsize=0.16]{-}(4,4)(2.5,2.5)
\psline[arrowsize=0.16]{-}(4,4)(4,2.5)
\psline[arrowsize=0.16]{-}(4,4)(5.5,2.5)

\psline[arrowsize=0.16]{-}(2.5,2.5)(2,1)
\psline[arrowsize=0.16]{-}(2.5,2.5)(2.5,1)
\psline[arrowsize=0.16]{-}(2.5,2.5)(3,1)

\psline[arrowsize=0.16]{-}(4,2.5)(3.5,1)
\psline[arrowsize=0.16]{-}(4,2.5)(4,1)
\psline[arrowsize=0.16]{-}(4,2.5)(4.5,1)

\psline[arrowsize=0.16]{-}(5.5,2.5)(5,1)
\psline[arrowsize=0.16]{-}(5.5,2.5)(5.5,1)
\psline[arrowsize=0.16]{-}(5.5,2.5)(6,1)

\psline[arrowsize=0.16, linestyle=dashed, dash=1pt 2pt]{-}(2,1)(1.9,0.7)
\psline[arrowsize=0.16, linestyle=dashed, dash=1pt 2pt]{-}(2.5,1)(2.5,0.7)
\psline[arrowsize=0.16, linestyle=dashed, dash=1pt 2pt]{-}(3,1)(3.1,0.7)

\psline[arrowsize=0.16, linestyle=dashed, dash=1pt 2pt]{-}(3.5,1)(3.4,0.7)
\psline[arrowsize=0.16, linestyle=dashed, dash=1pt 2pt]{-}(4,1)(4,0.7)
\psline[arrowsize=0.16, linestyle=dashed, dash=1pt 2pt]{-}(4.5,1)(4.6,0.7)

\psline[arrowsize=0.16, linestyle=dashed, dash=1pt 2pt]{-}(5,1)(4.9,0.7)
\psline[arrowsize=0.16, linestyle=dashed, dash=1pt 2pt]{-}(5.5,1)(5.5,0.7)
\psline[arrowsize=0.16, linestyle=dashed, dash=1pt 2pt]{-}(6,1)(6.1,0.7)

\psline[arrowsize=0.16]{-}(10,4)(8.5,2.5)
\psline[arrowsize=0.16]{-}(10,4)(10,2.5)
\psline[arrowsize=0.16]{-}(10,4)(11.5,2.5)

\psline[arrowsize=0.16]{-}(8.5,2.5)(8,1)
\psline[arrowsize=0.16]{-}(8.5,2.5)(8.5,1)
\psline[arrowsize=0.16]{-}(8.5,2.5)(9,1)

\psline[arrowsize=0.16]{-}(10,2.5)(9.5,1)
\psline[arrowsize=0.16]{-}(10,2.5)(10,1)
\psline[arrowsize=0.16]{-}(10,2.5)(10.5,1)

\psline[arrowsize=0.16]{-}(11.5,2.5)(11,1)
\psline[arrowsize=0.16]{-}(11.5,2.5)(11.5,1)
\psline[arrowsize=0.16]{-}(11.5,2.5)(12,1)

\psline[arrowsize=0.16, linestyle=dashed, dash=1pt 2pt]{-}(8,1)(7.9,0.7)
\psline[arrowsize=0.16, linestyle=dashed, dash=1pt 2pt]{-}(8.5,1)(8.5,0.7)
\psline[arrowsize=0.16, linestyle=dashed, dash=1pt 2pt]{-}(9,1)(9.1,0.7)

\psline[arrowsize=0.16, linestyle=dashed, dash=1pt 2pt]{-}(9.5,1)(9.4,0.7)
\psline[arrowsize=0.16, linestyle=dashed, dash=1pt 2pt]{-}(10,1)(10,0.7)
\psline[arrowsize=0.16, linestyle=dashed, dash=1pt 2pt]{-}(10.5,1)(10.6,0.7)

\psline[arrowsize=0.16, linestyle=dashed, dash=1pt 2pt]{-}(11,1)(10.9,0.7)
\psline[arrowsize=0.16, linestyle=dashed, dash=1pt 2pt]{-}(11.5,1)(11.5,0.7)
\psline[arrowsize=0.16, linestyle=dashed, dash=1pt 2pt]{-}(12,1)(12.1,0.7)

\end{pspicture}
\caption{A picture of the tree $\mathcal{T}_{3,2}$.}
\label{figure1}
\end{figure}

\subsection{The group $\aaut$ of almost automorphisms of $\tree$}

Recall that the group $\aut$ of automorphisms of the rooted tree $\tree$ is defined as the group of bijections of the set of vertices fixing the root and preserving the edges. In particular every automorphism of $\tree$ induces a homeomorphism of $\bord$. We now introduce a larger subgroup of the homeomorphism group of $\bord$, namely the group of homeomorphisms of $\bord$ which are piecewise tree automorphisms.

\begin{defi}
A finite subtree $T$ of $\tree$ is a rooted complete subtree if it contains the root as a vertex of degree $k$ and if any other vertex which is not a leaf has degree $d+1$.
\end{defi}

If $T$ is a finite rooted complete subtree of $\tree$ then its complement is a forest composed of finitely many copies of the rooted $d$-regular tree $\mathcal{T}_d$. If $T,T'$ are subtrees of $\tree$, a map $\psi: \tree \setminus T \rightarrow  \tree \setminus T'$ will be called a forest isomorphism if it maps each connected component of $\tree \setminus T$ onto a connected component of $\tree \setminus T'$, and induces a tree isomorphism on each of these connected components. Note that such a forest isomorphism naturally induces a homeomorphism of $\bord$.

\begin{defi}
The group $\aaut$ is defined as the set of equivalence classes of triples $(\psi,T,T')$, where $T,T'$ are finite rooted complete subtrees such that $| \partial T | = | \partial T' |$ and $\psi: \tree \setminus T \rightarrow  \tree \setminus T'$ is a forest isomorphism, where two triples are said to be equivalent if they give rise to the same homeomorphism of $\bord$. The multiplication in $\aaut$ is inherited from the composition in $\mathrm{Homeo}(\bord)$.
\end{defi}

We mention the following result, whose proof is easy and left to the reader, which gives an alternative definition of the group of almost automorphisms $\aaut$.

\begin{lem}
Let $T_1,T_1',T_2,T_2'$ be finite complete rooted subtrees of $\tree$. Then two triples $(\psi_1,T_1,T_1')$, $(\psi_2,T_2,T_2')$ are equivalent if and only if there exist finite rooted complete subtrees $T,T'$ so that $T$ (resp.\ $T'$) contains both $T_1$ and $T_2$ (resp.\ $T_1'$ and $T_2'$) and $\psi_1, \psi_2 : \tree \setminus T \rightarrow  \tree \setminus T'$ are equal.
\end{lem}

\begin{rmq}
By the previous lemma, when considering a triple $(\psi,T,T')$ representing an element of $\mathrm{AAut}(\tree)$, we can always assume that $T$ and $T'$ both contain a given finite subtree of $\tree$.
\end{rmq}

Note that since the only automorphism of $\tree$ inducing the trivial homeomorphism on $\bord$ is the identity, the group $\aaut$ contains a copy of the group $\aut$ of automorphisms of the tree $\tree$.

\section{Higman-Thompson groups} \label{sec-Hig}

\subsection{Introduction}

R.~Thompson introduced in 1965 three groups $F \leq T \leq V$, an introduction to which can be found in \cite{CFP}, while constructing a finitely generated group with unsolvable word problem. The groups $T$ and $V$ turned out to be the first examples of finitely presented infinite simple groups. Higman then generalized Thompson group $V$ to an infinite family of groups (which were originally denoted by $G_{d,k}$, but we will use the notation $V_{d,k}$ to keep in mind the analogy with Thompson group $V$, which is nothing else than $V_{2,1}$). K.~Brown later generalized Higman's construction to an infinite family of groups $F_{d,k} \leq T_{d,k} \leq V_{d,k}$ such that $F_{2,1} \simeq F$ and $T_{2,1} \simeq T$.

These groups were originally defined as automorphism groups of certain free algebras. We refer the reader to \cite{Brown} for an introduction from this point of view. 

The definition of the groups $V_{d,k}$ we give below is in terms of homeomorphism groups of the boundary of the quasi-regular rooted tree $\tree$. From this point of view, elements of these groups can be represented either as homeomorphisms of $\bord$ or by combinatorial diagrams, and we will use the interplay between these two representations.

\subsection{Higman-Thompson groups $V_{d,k}$ as subgroups of $\aaut$}

The definition that we adopt here for Higman-Thompson groups $V_{d,k}$ is borrowed from \cite{germs}. Note that when $k=d$ the group $V_{d,d}$ will be denoted by $V_{d}$ for simplicity.

\begin{defi}
An element of $\aaut$ is called locally order-preserving if it can be represented by a triple $(\psi,T,T')$ such that $T,T'$ are complete rooted subtrees of $\tree$ and $\psi: \tree \setminus T \rightarrow  \tree \setminus T'$ preserves the order of the boundary at infinity on each connected component.
\end{defi}

It follows from the order-preserving condition that such a forest isomorphism $\psi$ is uniquely determined by the induced bijection between the leaves of $T$ and the leaves of $T'$. Locally order-preserving almost automorphisms are easily checked to form a subgroup of $\aaut$.

\begin{defi}
The Higman-Thompson group $V_{d,k}$ is defined as the subgroup of $\aaut$ of locally order-preserving elements.
\end{defi}

Every locally order-preserving $v \in V_{d,k}$ has a unique representative $(\psi,T,T')$ so that $T,T'$ are complete rooted subtrees of $\tree$ and $T$ is minimal for the inclusion. Then $T'$ is also minimal and $(\psi,T,T')$ will be called the canonical representative of $v$. This notion coincides with the classical notion of reduced tree pair diagrams commonly used to study Thompson groups. The tree $T$ will be called the domain tree of $v$ and $T'$ the range tree. When considering a triple representing a locally order-preserving element, we will without further mention assume that this is the canonical representative.

%The planarity of $\tree$ induces a canonical way of ordering the leaves, say from left to right, of any finite rooted complete subtree of $\tree$. Following K.~Brown, we let $F_{d,k}$ (resp.\ $T_{d,k}$) be the subgroup of $V_{d,k}$ consisting of elements whose canonical representative $(\psi,T,T')$ is such that $\psi: \tree \setminus T \rightarrow  \tree \setminus T'$ preserves the order (resp.\ cyclic order) of the leaves. 

The following finiteness result is due to Higman \cite{Higman} (see also \cite{Brown}).

\begin{thm} \label{fp}
Higman-Thompson groups $V_{d,k}$ are finitely presented.
\end{thm}

This result will be used in Section \ref{sec-pres}, where we will enlarge a finite presentation of $V_{d,k}$ to obtain a compact presentation of the group $\Daaut$.

\subsection{Saturated subsets}

We now introduce a notion of saturated subsets inside the group $V_{d,k}$, needed in Section \ref{sec-pres}. We would like to point out that this notion is not necessary if one just wants to prove Theorem \ref{thma}. However it will be used in the proof of Theorem \ref{thmb} to perform the cost estimates carefully.

\begin{defi}
A subset $\Sigma \leq V_{d,k}$ is said to be saturated if for every $\sigma = (\psi,T,T') \in \Sigma$ and every $u \in \aut$, all the elements of $V_{d,k}$ having a canonical representative of the form $(\psi',u(T),T')$ belong to $\Sigma$.
\end{defi}

\begin{lem} \label{sat}
Every finite subset $\Sigma \leq V_{d,k}$ is contained in a finite saturated subset.
\end{lem}

\begin{proof}
Let $\Omega$ be the subset of $V_{d,k}$ consisting of elements of the form $(\psi,u(T),T')$, where $u \in \aut$ and $T$ is the domain tree of the canonical representative of some element of $\Sigma$. Clearly $\Omega$ contains $\Sigma$ and is saturated. Since $\Sigma$ is finite, the number of such trees $T$ is finite, and so is the set of $u(T)$, $u \in \aut$. The result then follows from the observation that, if $T$ is a fixed finite complete rooted subtree of $\tree$, then there are only finitely many elements of $V_{d,k}$ having a canonical representative of the form  $(\psi,T,T')$.
\end{proof}

\subsection{A lower bound for the word metric in $V_{d,k}$}

Here we give a lower bound for the word metric in the group $V_{d,k}$ in terms of a combinatorial data contained in the diagrams $(\psi,T,T')$ representing elements of $V_{d,k}$.

Recall that a $d$-caret (or caret for short) in $\tree$ is a subtree spanned by a vertex of level $n \geq 1$ and its $d$ neighbours of level $n+1$. We insist on the fact that we do not consider the subtree of $\tree$ spanned by the root and its $k$ neighbours as a caret. If $T$ is a finite complete rooted subtree of $\tree$ with $\kappa$ carets, then the number of leaves of $T$ is $(d-1) \kappa + k$. In particular if $v \in V_{d,k}$ has canonical representative $(\psi,T,T')$, then $T,T'$ have the same number of leaves, and consequently they also have the same number of carets. By abuse we will call it the number of carets of $v$ and denote it by $\kappa(v)$.

Metric properties of Higman-Thompson groups of type $F$ and $T$ can be essentially understood in terms of the number of carets of tree diagrams, the latter being quasi-isometric to the word-length associated to some finite generating set. The use of this point of view shed light on some interesting large scale geometric properties of these groups (see \cite{Burillo,Bur-Cle-Ste,Bur-Cle-Ste-Tab}). However metric properties of Higman-Thompson groups of type $V$ are far less well understood, as it follows from the work of Birget \cite{Birget} that the number of carets is no longer quasi-isometric the the word-length in Thompson group $V$. 

Nevertheless, the following lemma gives a lower bound for the word metric in $V_{d,k}$ in terms of the number of carets. Note that the same result appears in \cite{Birget} for the case of Thompson group $V$.

\begin{prop} \label{carets}
For any finite generating set $\Sigma$ of $V_{d,k}$, there exists a constant $C_{\Sigma} > 0$ such that for any $v \in V_{d,k}$, we have $\kappa(v) \leq C_{\Sigma} |v|_{\Sigma}$.
\end{prop}

\begin{proof}
Define $C_{\Sigma} = \max_{\sigma \in \Sigma} \kappa(\sigma)$. Now remark that when multiplying, say on the right, an element $v \in V_{d,k}$ by an element $\sigma \in \Sigma$, we obtain an element $v\sigma$ having a canonical representative with trees having at most $\kappa(v) + C_{\Sigma}$ carets. This is because when expanding the domain tree of $v$ to get a common expansion with the range tree of $\sigma$, we have to add at most $C_{\Sigma}$ carets. So it follows from a straightforward induction that every element of length at most $n$ with respect to the word metric associated to $\Sigma$ has a canonical representative with at most $C_{\Sigma} n$ carets, and the proof is complete.
\end{proof}

\section{The almost automorphism groups $\Daaut$} \label{aaut-type}

Almost automorphisms of $\tree$ are homeomorphisms of the boundary $\bord$ which are piecewise tree automorphisms. In this section we introduce a family of subgroups of $\aaut$ consisting of almost automorphisms which are piecewise tree automorphisms of a given type.

\subsection{Almost automorphisms of type $W(D)$}

Let $D \leq \mathrm{Sym}(d)$ be a subgroup of the symmetric group on $d$ elements. Define recursively $D_1 = D$, seen as the subgroup of the automorphism group of the rooted $d$-regular tree $\mathcal{T}_d$ acting on level one; and $D_{n+1} = D \wr D_n$ for every $n \geq 1$, where the permutational wreath product is associated with the natural action of $D_n$ on the set of vertices of level $n$ of $\mathcal{T}_d$. We now let $W(D)$ be the closed subgroup generated by the family $(D_n)$. For any $k \geq 1$ we let $W_k(D)$ be the subgroup of $\aut$ fixing pointwise the first level and acting by an element of $W(D)$ in each subtree rooted at level one. The group $W_k(D)$ is naturally isomorphic to the product of $k$ copies of the group $W(D)$.

\begin{defi}
An almost automorphism of $\tree$ is said to be piecewise of type $W(D)$ if it can be represented by a triple $(\psi,T,T')$ such that $T,T'$ are finite rooted complete subtrees of $\tree$ and $\psi: \tree \setminus T \rightarrow  \tree \setminus T'$ belongs to $W(D)$ on each connected component, after the natural identification of each connected component of $\tree \setminus T$ and $\tree \setminus T'$ with $\mathcal{T}_d$.
\end{defi}

We observe that by construction of $W(D)$, if a triple $(\psi_1,T_1,T_1')$ is such that $\psi_1: \tree \setminus T_1 \rightarrow  \tree \setminus T_1'$ belongs to $W(D)$ on each connected component, then for any equivalent triple $(\psi_2,T_2,T_2')$ such that $T_2$ (resp.\ $T_2'$) contains $T_1$ (resp.\ $T_1'$), then $\psi_2: \tree \setminus T_2 \rightarrow  \tree \setminus T_2'$ belongs to $W(D)$ on each connected component.

\begin{prop}
The set of almost automorphisms $\Daaut$ which are piecewise of type $W(D)$ is a subgroup of $\aaut$. 
\end{prop}

\begin{proof}
The only non-trivial fact that one needs to check is that $\Daaut$ is closed under multiplication, but this follows from the previous observation and from the fact that $W(D)$ is a subgroup of $\mathrm{Aut}(\mathcal{T}_d)$.
\end{proof}

%It is not hard to check that if $k' \equiv k \mod (d-1)$ then the groups $\mathrm{AAut}_{D}(\mathcal{T}_{d,k'})$ and $\Daaut$ are isomorphic (see \cite{Higman} for the same result for the groups $V_{d,k}$).

If $D$ is the full permutation group $\mathrm{Sym}(d)$ then $W(D) = \mathrm{Aut}(\mathcal{T}_d)$ and $\Daaut = \aaut$. At the opposite extreme, if $D$ is the trivial group then being piecewise trivial means being locally order-preserving and $\Daaut = V_{d,k}$. It is straightforward from the definition that if $D'$ contains $D$ then $\mathrm{AAut}_{D'}(\mathcal{T}_{d,k})$ contains $\Daaut$. In particular we note that for every subgroup $D \leq \mathrm{Sym}(d)$, the group $\Daaut$ always contains $V_{d,k}$.

\subsection{Topology on $\Daaut$}

By definition the group $\Daaut$ also contains a copy of the tree automorphism group $W_k(D)$. The latter comes equipped with a natural group topology, which is totally disconnected and compact, defined by saying that the pointwise level stabilizers form a basis of neighbourhoods of the identity. We would like to extend this topology to the group $\Daaut$, i.e.\ define a group topology on $\Daaut$ for which the subgroup $W_k(D)$ is an open subgroup. For this, let us first recall the following well-known lemma, a proof of which can be consulted in \cite[Chapter 3]{Bourb}.

\begin{lem} \label{bourb}
Let $G$ be a group and let $\mathcal{F}$ be a family of subgroups of $G$ which is filtering, i.e.\ so that the intersection of any two elements of $\mathcal{F}$ contains an element of $\mathcal{F}$. Assume moreover that for every $g \in G$ and every $U \in \mathcal{F}$, there exists $V \in \mathcal{F}$ so that $V \subset gUg^{-1}$. Then there exists a (unique) group topology on $G$ for which $\mathcal{F}$ is a base of neighbourhoods of the identity. 
\end{lem}

%Recall that a subgroup $H$ of a group $G$ is said to be commensurated by a subset $K$ of $G$ if for every $k \in K$, the subgroup $kHk^{-1} \cap H$ has finite index in both $H$ and $kHk^{-1}$. The following easy lemma, whose proof is left to the reader, provides an easy way to check commensurability.

%\begin{lem} \label{commens-gen}
%Let $G$ be a group and $S$ a generating set of $G$. Then a subgroup $H$ of $G$ is commensurated by $G$ if and only if it is commensurated by $S$.
%\end{lem}

%Now let $\mathcal{F}$ be the family of open subgroups of $W_k(D)$, which is a base of neighbourhoods of the identity in $W_k(D)$. It follows from Lemma \ref{bourb} that if $G$ is a group containing $W_k(D)$ as a subgroup, then there exists a group topology on $G$ for which $W_k(D)$ is an open subgroup as soon as $W_k(D)$ is commensurated in $G$.

%Now remark that Lemma \ref{commens-gen} together with Proposition \ref{decomp1} (a corollary of which is that $\Daaut$ is generated by $W_k(D)$ and $V_{d,k}$) imply that $W_k(D)$ is commensurated by $\Daaut$, because it is trivially commensurated by itself and commensurated by $V_{d,k}$ by Lemma \ref{commens}. We therefore obtain:

Now let $\mathcal{F}$ be the family of open subgroups of $W_k(D)$, which is a base of neighbourhoods of the identity in $W_k(D)$. By combining Lemma \ref{commens} together with Proposition \ref{decomp1} below, we see that the assumption of Lemma \ref{bourb} is satisfied in $G = \Daaut$, so we deduce:

\begin{prop}
There exists a group topology on $\Daaut$ making the inclusion $W_k(D) \hookrightarrow \Daaut$ continuous and open. In particular $\Daaut$ is a t.d.l.c.\ group (which is discrete if and only if $W_k(D)$ is trivial, if and only if $D$ is trivial).
\end{prop}

\begin{rmq}
\begin{enumerate}
\item [1)] It is interesting to point out that whereas the topology on $\aut$ coincides with the compact-open topology induced from $\mathrm{Homeo}(\bord)$, this is no longer true for the group $\aaut$. Indeed, the inclusion $\aaut \hookrightarrow \mathrm{Homeo}(\bord)$ is continuous but has a non-closed image. In other words, the topology on $\aaut$ is strictly finer than the compact-open topology. Actually the image of $\aaut \hookrightarrow \mathrm{Homeo}(\bord)$ is even dense, because one can check that the group $V_{d,k}$ is a dense subgroup of the homeomorphism group of $\bord$ with respect to the compact-open topology. \\
\item [2)] We also insist on the fact that for any permutation group $D$, the inclusion $\Daaut \hookrightarrow \aaut$ is always continuous, but its image is never closed unless $D$ is the full permutation group $\mathrm{Sym}(d)$. Indeed, $\Daaut$ contains the subgroup $V_{d,k}$ which is dense in $\aaut$ by Remark \ref{dense}, and therefore $\Daaut$ is never closed inside $\aaut$ unless it is the whole group.
\end{enumerate}
\end{rmq}

\subsection{Preliminaries} \label{sec-prel}

In this section we establish preliminary results about the groups $\Daaut$. Recall that $D \leq \mathrm{Sym}(d)$ is a finite permutation group, and we define recursively a family of finite subgroups of $\mathrm{Aut}(\mathcal{T}_d)$ by $D_1 = D$ and $D_{n+1} = D \wr D_n$ for every $n \geq 1$, where the permutational wreath product is associated with the natural action of $D_n$ on the $d^n$ vertices of level $n$ of $\mathcal{T}_d$. We denote by $D_{\infty}$ the subgroup generated by the family $(D_n)$ (which also coincides with the increasing union of the family $(D_n)$) and by $W(D)$ the closure of $D_{\infty}$ in $\mathrm{Aut}(\mathcal{T}_d)$. For $n \in \left\{1, \ldots, \infty \right\}$ we will also denote by $D_{n}^k$ the subgroup of $W_k(D)$ fixing pointwise the first level of $\tree$ and acting by an element of $D_{n}$ on each subtree rooted at the first level. Note that these groups have a natural decomposition $D_{n}^k = D_{n}^{(1)} \times \ldots \times D_{n}^{(k)}$, where $D_{n}^{(i)}$ is the subgroup of elements acting only on the $i$th subtree rooted at level one. 

If $\sigma = (\psi,T,T') \in V_{d,k}$, we let $W_k(D)_{\sigma}$ be the subgroup of $W_k(D)$ consisting of automorphisms which are the identity on the subtree $T$. Note that $W_k(D)_{\sigma}$ always contains some neighbourhood of the identity, and is consequently an open subgroup of $W_k(D)$. %The latter being compact, we obtain that $W_k(D)_{\sigma}$ is a finite index subgroup of $W_k(D)$.

\begin{lem} \label{commens}
For every $\sigma \in V_{d,k}$, we have the inclusion \[ \sigma \, W_k(D)_{\sigma} \, \sigma^{-1} \subset W_k(D). \] In particular $W_k(D)$ is commensurated by $V_{d,k}$.
\end{lem}

\begin{proof}
If $\sigma = (\psi,T,T')$ and $u \in W_k(D)_{\sigma}$, the reader will easily check that the element  $\sigma u \sigma^{-1} \in \Daaut$ is represented by a triple $(\psi',T',T')$, where $\psi'$ permutes trivially the connected components of $\tree \setminus T'$. Now if we consider the tree automorphism $u' \in W_k(D)$ being the identity on $T'$ and acting on $\tree \setminus T'$ like $\psi'$, it is clear that $u'$ is represented by the triple $(\psi',T',T')$, and therefore $\sigma u \sigma^{-1} = u' \in W_k(D)$.
\end{proof}

The next result yields a decomposition of the group $\Daaut$ in terms of the two subgroups $W_k(D)$ and $V_{d,k}$. It will be essential for proving Theorem \ref{thma} and Theorem \ref{thmb}.

\begin{prop} \label{decomp1}
For any $g \in \Daaut$ there exists $(u,v) \in W_k(D) \times V_{d,k}$ such that $g = uv$.
\end{prop}

\begin{proof}
Let $(\psi,T,T')$ be a triple representing $g \in \Daaut$. Let us consider the element $v \in \Daaut$ represented by the triple $(\xi,T,T')$ where $\xi$ is defined by declaring that each tree of the forest $\tree \setminus T$ is globally sent onto its image by $\psi$, but so that $\xi$ is order-preserving on each connected component of $\tree \setminus T$. Clearly we have $v \in V_{d,k}$. Now the discrepancy between $g$ and $v$ can be filled by performing the rooted tree automorphism induced by $g$ on each subtree rooted at a leaf of $T'$. But all of these can be achieved at the same time by an element of $W_k(D)$, namely the automorphism being the identity on $T'$ and acting as the desired rooted tree automorphism on each connected component of $\tree \setminus T'$.
\end{proof}

\begin{rmq} \label{dense}
Actually in Proposition \ref{decomp1}, $W_k(D)$ can be replaced by the pointwise stabilizer of the $n$th level of $\tree$ in $W_k(D)$, for every $n \geq 1$, the proof being the same. It yields in particular that $V_{d,k}$ is a dense subgroup of $\Daaut$.
\end{rmq}

Now given $g \in \Daaut$, there is not a unique $(u,v) \in W_k(D) \times V_{d,k}$ such that $g = uv$ because the two subgroups $W_k(D)$ and $V_{d,k}$ have a non-trivial intersection (as soon as $D$ is non-trivial). The measure of how this decomposition fails to be unique naturally leads to the study of the intersection of these two subgroups.

\begin{lem} \label{lem-inter}
The intersection between $V_{d,k}$ and $W_k(D)$ in $\Daaut$ is $D_{\infty}^k$. 
\end{lem}

\begin{proof}
$D_{n}^k$ lies inside $V_{d,k}$ and $W_k(D)$ for any $n \geq 1$, so the inclusion $D_{\infty}^k \subset V_{d,k} \cap W_k(D)$ is clear. To prove the reverse inclusion, let $g$ be an element of $V_{d,k} \cap W_k(D)$. Such an element $g$ is an automorphism of $\tree$ and therefore does act on the tree fixing setwise each level, so it is enough to prove that there exists an element of $D_{\infty}^k$ acting like $g$ on $\tree$. Since $g \in W_k(D)$, for every $n \geq 1$ there exists $g_n \in D_{n}^k$ acting like $g$ on the first $n$ levels of $\tree$. But now since $g \in V_{d,k}$, it is eventually order-preserving and therefore $g=g_n$ for $n$ large enough, which completes the proof.
\end{proof}

%Preuve initiale
%\begin{proof}
%For any $\Sigma \subset V_{d,k}$ we denote by $\Sigma'$ the set of elements of $V_{d,2}$ which are represented by a triple $(\phi,u^{-1}(T),T'')$, where $T$ ranges over the subtrees of $\mathcal{T}_d$ so that some triple $(\psi,T,T')$ represents an element of $\Sigma$, and $u$ ranges over $U_0$. If $\Sigma$ is finite then $T$ ranges over a finite set of subtrees and so does $u^{-1}(T)$, so by Remark \ref{nbfini} the subset $\Sigma' \subset V_{d,2}$ is finite.

%Let $\sigma \in \Sigma$ being fixed and let $(\psi,T,T')$ be a triple representing it. If $u \in U_0$ then $\sigma u$ is represented by a triple $(\psi',u^{-1}(T),T')$. Post-composing $\sigma u$ by an element of $U_0$, denoted by $u'^{-1}$, which acts as the identity on $T'$, and by restoring the order which had been broken by $u$ on each rooted tree of $\mathcal{T}_d \setminus T'$, we obtain an element $u'^{-1}\sigma u \in V_{d,2}$ represented by a triple $(\psi'',T_u,T'')$, and therefore $u'^{-1}\sigma u \in \Sigma'$.
%\end{proof}

%Roughly, the idea of Lemma \ref{increase} is to find a finite set of elements $\Delta \leq V_{d,k}$, so that given an element $u \in D_{\infty}^{(i)}$, we can find $\delta \in \Delta$ so that conjugating by $\delta$ increases by one the level of the action of $u$.

The end of this section is devoted to establishing Lemma \ref{increase}, which will be applied in the proof of Lemma \ref{lem-i} in Section \ref{sec-pres}.

Recall that $(a_1, \ldots a_k)$ are the ordered vertices of level one of $\tree$, and that $\tree^{a_i}$ is the full subtree of $\tree$ rooted at $a_i$. Recall also that $D_{\infty}^{(i)}$ are the automorphisms of $\tree$ acting by an element of $D_{\infty}$ on $\tree^{a_i}$, and acting trivially outside $\tree^{a_i}$. Finally recall that we denote by $(a_ib_1, \ldots, a_ib_d)$ the ordered neighbours of $a_i$ in $\tree^{a_i}$. In what follows, by convention indexes will be taken modulo $k$ (for example $a_{k+1}$ will denote the vertex $a_1$).

For every $i=1 \ldots k$ and $j=1 \ldots d$, we define an element $\delta_{i,j} = (\psi,T,T') \in V_{d,k}$ by the following manner: \begin{itemize} 
\item $T$ is the smallest finite complete rooted subtree containing the $d$ descendants of $a_i$;
\item $T'$ is the smallest finite complete rooted subtree containing the $d$ descendants of $a_{i+1}$; 
\item $\psi$ is defined by the formulas
\begin{itemize} 
\item $\psi(a_{\ell}) = a_{\ell}$ for every $\ell \notin \left\{ i,i+1\right\}$;
\item $\psi(a_{i+1}) = a_{i+1}b_j$; 
\item $\psi(a_ib_{\ell}) = a_{i+1}b_{\ell}$ for every $\ell \neq j$;
\item $\psi(a_ib{j}) = a_i$.
\end{itemize}
\end{itemize}
For example the diagram of $\delta_{1,j}$ is represented in Figure \ref{figure-delta} in the case $k=2$. We denote by $\Delta$ the set of $\delta_{i,j}$, for $i = 1 \ldots k$, $j = 1 \ldots d$.

\begin{figure}
\psset{unit=23pt}
\begin{pspicture}(12,4)
\psline[arrowsize=0.16]{-}(2,2.5)(3,3.5)
\psline[arrowsize=0.16]{-}(3,3.5)(4,2.5)
\psdot(4,2.5)
\rput(4,2){$d+1$} 
\psline[arrowsize=0.16]{->}(5.25,2)(6.75,2)
\rput(6,2.5){$\psi$}
\psline[arrowsize=0.16]{-}(8,2.5)(9,3.5)
\psline[arrowsize=0.16]{-}(9,3.5)(10,2.5)
\psdot(8,2.5)
\rput(8,2){$j$}

\psline[arrowsize=0.16]{-}(2,2.5)(1.2,1)
\psdot(1.2,1)
\rput(1.2,0.5){$1$}
\psline[arrowsize=0.16]{-}(2,2.5)(2,1)
\psdot(2,1)
\rput(2,0.5){$j$}
\psline[arrowsize=0.16]{-}(2,2.5)(2.8,1)
\psdot(2.8,1)
\rput(2.8,0.5){$d$}
\rput(1.72,1.3){$\cdots$}
\rput(2.35,1.3){$\cdots$}

\psline[arrowsize=0.16]{-}(10,2.5)(9.2,1)
\psdot(9.2,1)
\rput(9.2,0.5){$1$}
\psline[arrowsize=0.16]{-}(10,2.5)(10,1)
\psdot(10,1)
\rput(10,0.5){$d+1$}
\psline[arrowsize=0.16]{-}(10,2.5)(10.8,1)
\psdot(10.8,1)
\rput(10.8,0.5){$d$}
\rput(9.72,1.3){$\cdots$}
\rput(10.35,1.3){$\cdots$}
\end{pspicture}
\caption{The diagram of $\delta_{1,j}$ when $k=2$.}
\label{figure-delta}
\end{figure}
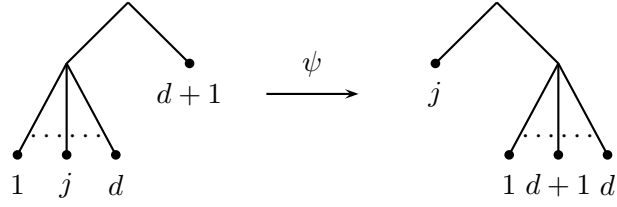

\begin{lem} \label{increase}
Let $i=1 \ldots k$, $j=1 \ldots d$, and let $n \geq 1$. Assume that $u \in D_{\infty}^{(i)}$ acts trivially outside $\tree^{a_ib_j}$, and $u$ has at most $n+1$ carets. Then the element $\delta_{i,j} u \delta_{i,j}^{-1}$ belongs to $D_{\infty}^{(i)}$ and has at most $n$ carets.
\end{lem}

\begin{proof}
This is a direct consequence of the fact that $\delta_{i,j}$ maps the subtree $\tree^{a_ib_j}$ onto the the subtree $\tree^{a_i}$, and is order-preserving on $\tree^{a_ib_j}$.
\end{proof}

\section{Presentation of $\Daaut$} \label{sec-pres}

In this section we write down an explicit presentation of the group $\Daaut$ for any $k \geq1, d \geq 2$ and $D \leq \mathrm{Sym}(d)$, and prove Theorem \ref{thma} and Theorem \ref{thmb}.

Let $\Sigma$ denote a finite generating set of the group $V_{d,k}$, which is supposed to contain $\Delta$ (recall that $\Delta$ has been defined at the end of Section \ref{aaut-type}). Enlarging $\Sigma$ if necessary, we can also assume that $\Sigma$ is saturated by Lemma \ref{sat}. This implies the following:

\begin{lem} \label{exchange}
We have the inclusion $\Sigma \, W_k(D) \subset W_k(D) \, \Sigma$.
\end{lem}

\begin{proof}
It follows from the proof of Proposition \ref{decomp1} that any $\sigma_1 u_1 \in \Sigma \, W_k(D)$ can be written $u_2 \sigma_2$ with $u_2 \in W_k(D)$ and $\sigma_2 \in V_{d,k}$ being of the form $(\psi,u^{-1}(T),T')$, where $T$ is the domain tree of $\sigma$. Since $\Sigma$ is saturated, $\sigma_2$ belongs to $\Sigma$ and therefore $\sigma_1 u_1 \in W_k(D) \, \Sigma$.
\end{proof}

According to Proposition \ref{decomp1}, the set $S = \Sigma \cup W_k(D)$ is a generating set of $\Daaut$. The strategy to prove Theorem \ref{thma} will be to list some particular relations between the elements of $S$ satisfied in the group $\Daaut$, and then to prove that they generate all the relations in $\Daaut$.

\begin{enumerate}
\item [$\left(R_{\Sigma}\right)$] According to Theorem \ref{fp} there exists a finite set of words $R_{\Sigma} \leq \Sigma^*$ so that $\left\langle \Sigma \mid R_{\Sigma} \right\rangle$ is a presentation of $V_{d,k}$.
\end{enumerate}

\begin{enumerate}
\item [$\left(R_{D}\right)$] We let $R_{D}$ be the set of words of the form $u_1u_2u_3^{-1}$, $u_i \in W_k(D)$, whenever the relation $u_1u_2=u_3$ is satisfied in the group $W_k(D)$.
\end{enumerate}

\begin{enumerate}
\item [$\left(R_{1}\right)$] The set of relations $R_{1}$ will correspond to commensurating relations in $\Daaut$. Recall that if $\sigma \in V_{d,k}$ and $u \in W_k(D)_{\sigma}$ then $\sigma u \sigma^{-1} \in W_k(D)$ by Lemma \ref{commens}. We let $R_{1}$ be the set of words of the form $\sigma u_1 \sigma^{-1} u_2^{-1}$, where $\sigma \in \Sigma, u_1 \in W_k(D)_{\sigma}, u_2 \in W_k(D)$, whenever the relation $\sigma u_1 \sigma^{-1} = u_2$ holds in $\Daaut$. 
\end{enumerate}

\begin{enumerate}
\item [$\left(R_{2}\right)$] We add relations corresponding to the fact that the subgroup $D_1^{k}$ of $\Daaut$ lies in the intersection of $V_{d,k}$ and $W_k(D)$. More precisely, for every $i \in \left\{1, \ldots, k\right\}$ and every $u \in D_1^{(i)}$, we choose a word $w_u \in \Sigma^*$ so that $u = w_u$ in $\Daaut$. We denote by $R_2$ the set of words $u w_u^{-1}$, and by $r_i$ the maximum word length of the words $w_u$ when $u$ ranges over $D_1^{(i)}$.
\end{enumerate}

\begin{enumerate}
\item [$\left(R_{3}\right)$]  By Lemma \ref{exchange}, for every $\sigma_1 \in \Sigma$ and $u_1 \in W_k(D)$ we can pick some $u_2 \in W_k(D)$ and $\sigma_2 \in \Sigma$ so that $\sigma_1 u_1 = u_2 \sigma_2$ in $\Daaut$. We denote by $R_3$ the set of words $\sigma_1 u_1 \sigma_2^{-1} u_2^{-1}$.
\end{enumerate}

Denote by $R = R_{\Sigma} \cup R_{D} \cup_i R_i$ the union of all these relations. Note that elements of $R$ have bounded length with respect to the compact generating set $S = \Sigma \cup W_k(D)$ of $\Daaut$. 

We let $G$ be the group defined by the presentation $\left\langle S \mid R \right\rangle$, that is we have a short exact sequence \[ 1 \rightarrow \mathcal{R} = \left\langle \! \left\langle R \right\rangle \! \right\rangle \rightarrow F_S \rightarrow G \rightarrow 1, \] where $F_S$ is the free group over the set $S$ and $\left\langle \! \left\langle R \right\rangle \! \right\rangle$ is the normal subgroup generated by $R$. Denote by $a: F_S \rightarrow \left[0, + \infty \right]$ the corresponding area function, which by definition associates to $w \in \mathcal{R}$ the least integer $n$ so that $w$ is a product of at most $n$ conjugates of elements of $R$, and $a(w) = + \infty$ if $w \notin \mathcal{R}$. We also define the associated cost function $c: F_S \times F_S \rightarrow \left[0, + \infty \right]$ by $c(w_1,w_2) = a(w_1^{-1}w_2)$. This function estimates the cost of converting $w_1$ to $w_2$, or the cost of going from $w_1$ to $w_2$, in the sense that $c(w_1,w_2)$ is the distance in $F_S$ between $w_1$ and $w_2$ with respect to the word metric associated to the union of conjugates of $R$. In particular the cost function is symmetric and satisfies the triangular inequality $c(w_1,w_3) \leq c(w_1,w_2) + c(w_2,w_3)$ for every $w_1,w_2,w_3 \in F_S$. This, combined with the bi-invariance of the cost function, yields the following inequality, which will be used repeatedly: for every $\ell \geq 1$ and every $w_1, \ldots, w_{\ell},w'_1,\ldots,w'_{\ell} \in F_S$, we have: \[ c(w_1 \ldots w_{\ell},w'_1 \ldots w'_{\ell}) \leq \sum_{i=1}^{\ell} c(w_i,w'_i). \]

Two words $w_1, w_2 \in F_S$ are said to be homotopic if they represent the same element of $G$, i.e.\ if $c(w_1,w_2) < + \infty$. %A word $w$ is said to be null-homotopic if it represents the identity, i.e.\ if $w \in \mathcal{R}$.

We are now able to state the main theorem of this section, which implies both Theorem \ref{thma} and Theorem \ref{thmb}.

\begin{thm} \label{presentation}
The natural map $G \rightarrow \Daaut$ is an isomorphism. Furthermore, the Dehn function of the presentation $\left\langle S \mid R \right\rangle$ is asymptotically bounded by that of $V_{d,k}$.
\end{thm}

It is clear that the map from $G$ to $\Daaut$ is a well defined morphism because relations $R_{\Sigma}, R_{D}, (R_i)$ are satisfied in $\Daaut$, and it is onto because $S$ generates the group $\Daaut$. So proving the first claim comes down to proving that this morphism is injective, i.e.\ any word in $F_S$ representing the identity in the group $\Daaut$ already represents the trivial element in the group $G$. This will be achieved, as well as the proof of the upper bound on the Dehn function, in Proposition \ref{dehn}, using both geometric and combinatorial arguments.

The goal of Proposition \ref{u-replace} is to prove that relations in the group $\Daaut$ coming from the fact that the subgroups $W_k(D)$ and $V_{d,k}$ intersect non-trivially, are already satisfied in the group $G$, and to obtain a precise estimate of their cost.

The following is the main technical lemma towards Proposition \ref{u-replace}. 

\begin{lem} \label{lem-i}
Fix $i \in \left\{1, \ldots, k\right\}$, and let $C_i = 2d + \max(2,r_i)$ (recall that $r_i$ has been defined with the set of relators $R_2$). Then for every $n \geq 0$ and every $u \in D_{\infty}^{(i)}$ having at most $n$ carets, there exists a word $w \in \Sigma^*$ of length at most $C_i n$ so that the relation $u=w$ holds in $G$ and has cost at most $C_i n$.
\end{lem}

\begin{proof}
We use induction on $n$. The result is trivially true for $n=0$ because the only element of $D_{\infty}^{(i)}$ with zero carets is the identity, and is true for $n=1$ thanks to the set of relators $R_2$.

The idea of the proof of the induction step is the following. Given $u \in D_{\infty}^{(i)}$ with at most $n+1$ carets, we begin by multiplying it by an element of $D_{1}^{(i)}$ in order to ensure that it acts trivially on the first level of $\tree^{a_i}$. The resulting automorphism has a natural decomposition into a product of $d$ elements of $D_{\infty}^{(i)}$, coming from its action on the subtrees $\tree^{a_ib_1}, \ldots, \tree^{a_ib_d}$, with a nice control on the number of carets of each element of this product. We then apply the induction hypothesis to each of these elements, after having reduced their number of carets by conjugating by an element of $\Delta$, which has the effect of increasing by $1$ the level of the subtree on which they act.

Henceforth we assume that $u \in D_{\infty}^{(i)}$ is an element having at most $n+1$ carets, with $n \geq 1$. If we let $\bar{u}$ denote the element of $D_{1}^{(i)}$ acting like $u$ on the first level of $\tree^{a_i}$, it is clear that $u' = u \bar{u}^{-1}$ stabilizes pointwise the first level of $\tree^{a_i}$. Using relators from $R_2$, we pick a word $w_{\bar{u}} \in \Sigma^*$ so that the relation $\bar{u} = w_{\bar{u}}$ holds in $G$ and has cost at most one.

Now in the group $\Daaut$, since $u'$ acts trivially on the first level of $\tree^{a_i}$, it has a natural decomposition $u' = u_{1} \ldots u_d$, where each $u_{\ell} \in D_{\infty}^{(i)}$ acts trivially outside the subtree $\tree^{a_ib_{\ell}}$. In the rest of the proof, we implicitly assume that we only consider integers $\ell \in \left\{1, \ldots, d\right\}$ such that $u_{\ell}$ is non-trivial. Note that each $u_{\ell}$ has at most $n+1$ carets and that \[ \sum_{\ell} \kappa(u_{\ell}) \leq \kappa(u) + d -1 \leq n+d,\] because the caret corresponding to the root of $\tree^{a_i}$ can appear $d$ times in this sum, whereas it is counted only once in $\kappa(u)$. Note also that thanks to the set of relators $R_D$, the relation $u' = u_1 \ldots u_d$ also holds in the group $G$ and has cost at most $d$.

Remark that by construction of the set $\Delta$, every element of $W_k(D)$ acting trivially on the second level of $\tree$ lies inside $W_k(D)_{\delta}$ for every $\delta \in \Delta$ (see Subsection \ref{sec-prel} for the definition of $W_k(D)_{\delta}$). In particular if $\ell \in \left\{1, \ldots, d\right\}$ and if $\delta_{\ell} = \delta_{i,\ell}$, we have $u_{\ell} \in W_k(D)_{\delta_{\ell} }$ and thanks to $R_1$, the word $\delta_{\ell} u_{\ell} \delta_{\ell}^{-1}$ represents in the group $G$ an element $\tilde{u_{\ell}} \in D_{\infty}^{(i)}$ with at most $\kappa(u_{\ell})-1$ carets according to Lemma \ref{increase}. Note in particular that \begin{equation} \label{eq1} \sum_{\ell} \kappa(\tilde{u_{\ell}}) \leq \sum_{\ell} (\kappa(u_{\ell})-1) \leq \sum_{\ell} \kappa(u_{\ell}) - d \leq n. \end{equation}

Now since $\kappa(\tilde{u_{\ell}}) \leq \kappa(u_{\ell})-1 \leq n$ for every $\ell \in \left\{1, \ldots, d\right\}$, we may apply the induction hypotheses to $\tilde{u_{\ell}}$, and we obtain a word $\tilde{w_{\ell}}$ of length at most $C_i \kappa(\tilde{u_{\ell}})$ so that $\tilde{u_{\ell}} = \tilde{w_{\ell}}$ in $G$ and $c(\tilde{u_{\ell}},\tilde{w_{\ell}}) \leq C_i \kappa(\tilde{u_{\ell}})$. If we denote by $w_{\ell} = \delta_{\ell}^{-1} \tilde{w_{\ell}} \delta_{\ell}$, then the relation $u_{\ell} = w_{\ell}$ holds in the group $G$ and has cost \[ c(u_{\ell},w_{\ell}) \leq c(u_{\ell},\delta_{\ell}^{-1} \tilde{u_{\ell}} \delta_{\ell}) + c(\delta_{\ell}^{-1} \tilde{u_{\ell}} \delta_{\ell},w_{\ell}) \leq 1 + C_i \kappa(\tilde{u_{\ell}}). \]

We now want to put all these pieces together and conclude the proof of the induction step. Consider the word $w = w_1 \ldots w_d w_{\bar{u}} \in \Sigma^*$. Its length is easily seen to satisfy \[ |w|_{\Sigma} \leq \sum_{\ell = 1}^d |w_{\ell}|_{\Sigma} + |w_{\bar{u}}|_{\Sigma} \leq \sum_{\ell = 1}^d (2 + |\tilde{w_{\ell}}|_{\Sigma}) + r_i \leq C_i \sum_{\ell = 1}^d \kappa(\tilde{u_{\ell}}) + 2d + r_i \leq C_i(n+1), \] because $\sum_{\ell} \kappa(\tilde{u_{\ell}}) \leq n$ according to (\ref{eq1}), and $C_i \geq 2d + r_i$. Furthermore, we claim that the relation $u=w$ is satisfied in $G$ and has cost at most $C_i(n+1)$. Indeed, by triangular inequality we have \[ \begin{aligned} c(u,w) & \leq c(u,u'\bar{u}) + c(u'\bar{u},w_1 \ldots w_d w_{\bar{u}}) \\ & \leq c(u,u'\bar{u}) + c(u',w_1 \ldots w_d) + c(\bar{u},w_{\bar{u}}). \end{aligned}\] Now $c(u,u'\bar{u})$ and $c(\bar{u},w_{\bar{u}})$ are at most $1$, and by using triangular inequality again we obtain \[ \begin{aligned} c(u,w) & \leq 2 + c(u',u_1 \ldots u_d) + c(u_1 \ldots u_d,w_1 \ldots w_d) \\ & \leq 2+d + \sum c(u_{\ell},w_{\ell}) \\ & \leq 2+d+\sum(1 + C_i \kappa(\tilde{u_{\ell}})) \\ & \leq 2+2d+C_i \sum \kappa(\tilde{u_{\ell}}) \\ & \leq  2+2d+C_i n \\ & \leq C_i(n+1), \end{aligned} \] so the proof of the induction step is complete.
\end{proof}

\begin{prop} \label{u-replace}
There exists a constant $C > 0$ such that for every $u \in D_{\infty}^k$, there exists a word $w \in \Sigma^*$ of length at most $C \kappa(u)$ so that the relation $u=w$ holds in $G$ and has cost at most $k + \kappa(u)$.
\end{prop}

\begin{proof}
Let $C = \max_i C_i$, where the constant $C_i$ is defined in Lemma \ref{lem-i}. Any $u \in D_{\infty}^k$ can be written $u=u_1 \ldots u_k$ in $\Daaut$, with $u_i \in D_{\infty}^{(i)}$ and $\kappa(u) = \kappa(u_1) + \cdots + \kappa(u_k)$. Applying Lemma \ref{lem-i} to $u_i$, we get a word $w_i$ of length at most $C_i \kappa(u_i)$ so that the relation $u_i=w_i$ holds in $G$ and has cost at most $\kappa(u_i)$. Let $w=w_1 \ldots w_k$. Then \[ |w|_{\Sigma} \leq \sum_{i=1}^k |w_i|_{\Sigma} \leq \sum_{i=1}^k C_i \kappa(u_i) \leq C \sum_{i=1}^k \kappa(u_i) = C\kappa(u).\] Moreover the relation $u=u_1 \ldots u_k$ holds in $G$ thanks to the set of relators $R_{D}$. Consequently in $G$ we have $u=w$ at a total cost of at most \[ \begin{aligned} c(u,u_1 \ldots u_k) + c(u_1 \ldots u_k,w) & \leq k+ \sum_{i=1}^k c(u_i,w_i) \\ & \leq k + \sum_{i=1}^k \kappa(u_i) = k + \kappa(u). \end{aligned}\]
\end{proof}

The next lemma will reduce the estimate of the area function to its estimate for words of the special form $W_k(D) \Sigma^*$. 

\begin{lem} \label{u-left}
There exists a constant $c_1>0$ such that for any $n$ and any word $w \in S^*$ of length at most $n$, there exists a word $w' = u \sigma_1 \ldots \sigma_j$ of length at most $n$, where $u \in W_k(D)$, $\sigma_1, \ldots, \sigma_j \in \Sigma$, so that $w'$ is homotopic to $w$ and $c(w,w') \leq c_1 n \log(n)$.
\end{lem}

\begin{proof}
For any word $w \in S^*$, define \[ \tau(w) = \inf \left\{c(w,w') : w' \in W_k(D) \Sigma^* \, \text{and $w'$ is homotopic to $w$} \right\}, \] and \[ f(n) =  \sup \left\{ \tau(w) : w \in S^* \, \text{has length at most $n$} \right\}. \] Note that both $\tau$ and $f$ take finite values thanks to relators from $R_3$ and $R_D$. We want to prove that $f(n) \leq c_1 n \log(n)$ for some constant $c_1$.

We use an algorithmic strategy. Given a word $w$, we first divide it into two subwords, then apply the algorithm to each of them and finally merge the results. More precisely, let us consider a word $w$ of length $2^{n+1}$, and divide it into two subwords $w_1, w_2$ of length $2^n$. By definition of the function $f$, there exists words $w'_1,w'_2 \in W_k(D) \Sigma^*$ such that $c(w_1,w'_1), c(w_2,w'_2) \leq f(2^n)$. Now in the word $\bar{w} = w'_1w'_2 \in W_k(D) \Sigma^* W_k(D) \Sigma^*$ we can move the $W_k(D)$ part of $w'_2$ to the left by applying at most $2^n-1$ relators of $R_3$, and merge it with the $W_k(D)$ part of $w'_1$ with cost $1$ thanks to the set of relators $R_D$. We therefore get a word $w' \in W_k(D) \Sigma^*$ homotopic to $w$ and so that $c(w,w') \leq 2 f(2^n) + (2^n-1) + 1$, which implies that $\tau(w) \leq 2 f(2^n) + 2^n$. By definition of $f$, we obtain $f(2^{n+1}) \leq 2 f(2^n) + 2^n$, from which we easily get the inequality $f(2^n) \leq n 2^{n-1}$. The result then follows from this inequality together with the fact that $f$ is non-decreasing.
\end{proof}

\begin{prop} \label{dehn}
There exists a constant $c>0$ such that if $w \in F_S$ has length at most $n$ and represents the identity in $\Daaut$, then $w$ already represents the identity in $G$ and has area \[a(w) \leq c n \log(n) +  \delta(cn),\] where $\delta$ is the Dehn function of the presentation $\left\langle \Sigma, R_{\Sigma} \right\rangle$ of $V_{d,k}$.
\end{prop}

\begin{proof}
We first apply Lemma \ref{u-left} to the word $w$, and obtain a word $w' = u \sigma_1 \ldots \sigma_j$ of length at most $n$ so that $c(w,w') \leq c_1 n \log(n)$. Since $w$ is trivial in $\Daaut$, so is $w'$, and therefore the element $u^{-1}$ belongs to $W_k(D) \cap V_{d,k} = D_{\infty}^k$. Moreover this element has length at most $n$ in the group $V_{d,k}$ because the word $w'$ has length at most $n$. According to Proposition \ref{carets}, we have $\kappa(u^{-1}) \leq C_{\Sigma} n$. Applying Proposition \ref{u-replace} to $u^{-1}$ yields a word $w'' \in \Sigma^*$ of length at most $C \kappa(u^{-1}) \leq C C_{\Sigma} n$ so that the relation $u^{-1}=w''$ holds in $G$ and has cost at most $k + C_{\Sigma} n$. Therefore we obtain that
\[ \begin{aligned} a(w) \leq & c(w,w') + c(u^{-1},w'') + c(w'',\sigma_1 \ldots \sigma_j) \\ \leq & c_1 n \log(n) + (k + C_{\Sigma} n) + \delta \left(C C_{\Sigma} n + n \right), \end{aligned}\]
because $w''$ and $\sigma_1 \ldots \sigma_j$ represent the same element in $V_{d,k}$ and $w''(\sigma_1 \ldots \sigma_j)^{-1}$ has length at most $C C_{\Sigma}n + n$, so $c(w'',\sigma_1 \ldots \sigma_j)$ is at most $\delta(C C_{\Sigma}n+n)$ by definition of the Dehn function. Therefore $a(w) \leq c n \log(n) +  \delta(cn)$ for some constant $c$ depending only on $\Sigma$, and the proof is complete.
\end{proof}

In particular we deduce from Proposition \ref{dehn} that the Dehn function of $\Daaut$ is $\preccurlyeq n \log n + \delta_{V_{d,k}}$. But now the group $V_{d,k}$ is not Gromov-hyperbolic since it has a $\mathbb{Z}^2$ subgroup, so its Dehn function is not linear and consequently at least quadratic \cite{Bow}. Therefore $n \log n \preccurlyeq \delta_{V_{d,k}}$, and the Dehn function of $\Daaut$ is thus asymptotically bounded by $\delta_{V_{d,k}}$.

\section{Compact presentability of \sch completions} \label{sec-sch}

If $\Gamma$ is a group with a commensurated subgroup $\Lambda$, the \sch completion process builds a t.d.l.c.\ group $\Gamma \rpc \Lambda$ and a morphism $\Gamma \rightarrow \Gamma \rpc \Lambda$, so that the image of $\Gamma$ is dense and the closure of the image of $\Lambda$ is compact open. It was formally introduced in \cite{tzanev}, following an idea appearing in \cite{schlichting}. 

We would like to point out that \sch completions are sometimes called \textit{relative profinite completions} \cite{SW, eld-wil}, but we choose not to use this terminology in order to avoid confusion with the notion of \textit{localised profinite completion} appearing in \cite{reid-pc}. Although we will not use this terminology, we also note that a group together with a commensurated subgroup is sometimes called a Hecke pair.

The main theorem of this section is a general result about compact presentability of \sch completions:

\begin{thm} \label{thm-sch}
Let $\Gamma$ be a finitely presented group and let $\Lambda$ be a finitely generated commensurated subgroup. Then the t.d.l.c.\ group $\Gamma \rpc \Lambda$ is compactly presented.
\end{thm} 

Before going into the proof, let us mention the following result which can derived from Theorem \ref{thm-sch}. As mentioned above, the notion of \sch completion is different but closely related to the notion of profinite completion of a group localised at a subgroup \cite{reid-pc}. More precisely, it is proved in \cite[Corollary 3, (vii)]{reid-pc} that the \sch completion $\Gamma \rpc \Lambda$ is the quotient of the profinite completion of $\Gamma$ localised at $\Lambda$ by a compact normal subgroup. Now since a locally compact group is compactly presented if and only if one of its quotients by a compact normal subgroup is, we obtain:

\begin{cor}
If $\Gamma$ is a finitely presented group with a finitely generated commensurated subgroup $\Lambda$, then the profinite completion of $\Gamma$ localised at $\Lambda$ is compactly presented.
\end{cor} 

\subsection{From commensurated subgroups to t.d.l.c.\ groups}

We start by recalling the definition of the process of \sch completion.

Let $\Gamma$ be a group and let $\Lambda$ be a subgroup of $\Gamma$. The left action of $\Gamma$ on the coset space $\Gamma / \Lambda$ yields a homomorphism $\Gamma \rightarrow \mathrm{Sym}(\Gamma / \Lambda)$, whose kernel is the normal core of $\Lambda$, i.e.\ the largest normal subgroup of $\Gamma$ contained in $\Lambda$ (or equivalently, the intersection of all conjugates of $\Lambda$). The \sch completion of $\Gamma$ with respect to $\Lambda$, denoted $\Gamma \rpc \Lambda$, is by definition the closure of the image of $\Gamma$ in $\mathrm{Sym}(\Gamma / \Lambda)$, the latter group being equipped with the topology of pointwise convergence.

Recall that $\Lambda$ is said to be commensurated by a subset $K$ of $\Gamma$ if for every $k \in K$, the subgroup $k \Lambda k^{-1} \cap \Lambda$ has finite index in both $\Lambda$ and $k \Lambda k^{-1}$. We say that $\Lambda$ is a commensurated subgroup if it is commensurated by the entire group $\Gamma$. It can be checked that if this holds, then the closure of the image of $\Lambda$ in $\Gamma \rpc \Lambda$ is a compact open subgroup. In particular $\Gamma \rpc \Lambda$ is a t.d.l.c.\ group. Note that by construction the image of $\Gamma$ in $\Gamma \rpc \Lambda$ is a dense subgroup. %For example if $\Lambda$ is trivial then $\Gamma \rpc \Lambda \simeq \Gamma$, and if $\Lambda$ is normal in $\Gamma$ then $\Gamma \rpc \Lambda \simeq \Gamma / \Lambda$.

From now $\Lambda$ will be a commensurated subgroup of a group $\Gamma$. We point out that although the map $\Gamma \rightarrow \Gamma \rpc \Lambda$ is generally not injective, for the sake of simplicity we still use the notation $\Lambda$ and $\Gamma$ for their images in the group $\Gamma \rpc \Lambda$.

The next two lemmas are straightforward, we provide proofs for completeness. 

\begin{lem} \label{decomp-rpc}
We have $\Gamma \rpc \Lambda = \overline{\Lambda} \cdot \Gamma$.
\end{lem}

\begin{proof}
Since $\overline{\Lambda}$ is an open subgroup of $\Gamma \rpc \Lambda$, $\overline{\Lambda} g$ is an open neighbourhood of $g$ for any $g \in \Gamma \rpc \Lambda$. Therefore the dense subgroup $\Gamma$ intersects $\overline{\Lambda} g$, meaning that there exist $\gamma \in \Gamma$ and $\lambda \in \overline{\Lambda}$ so that $\lambda g = \gamma$, i.e.\ $g = \lambda^{-1} \gamma$.
\end{proof}

\begin{lem} \label{inter-rpc}
The subgroups $\overline{\Lambda}$ and $\Gamma$ intersect along $\Lambda$.
\end{lem}

\begin{proof}
The subgroup $\Lambda$ stabilizes the coset $\Lambda$ in $\mathrm{Sym}(\Gamma / \Lambda)$, so every element of $\overline{\Lambda}$ must stabilizes this coset as well by definition of the topology. Therefore $\overline{\Lambda} \cap \Gamma \subset \Lambda$. The reverse inclusion is clear.
\end{proof}

The following result is a useful tool to identify some t.d.l.c.\ group $G$ with the \sch completion of one of its dense subgroups. It is due to Shalom and Willis \cite[Lemmas 3.5-3.6]{SW}. 

\begin{prop} \label{lem-sw}
Let $G$ be a topological group with a compact open subgroup $U$. If $\Gamma$ is a dense subgroup of $G$, then $\Gamma \cap U$ is commensurated in $\Gamma$ and the embedding of $\Gamma$ in $G$ induces an isomorphism of topological groups $\varphi: \Gamma \rpc (\Gamma \cap U) \rightarrow G / K_U$, where $K_U$ is the normal core of $U$. In particular if $U$ contains no non-trivial normal subgroup of $G$, then $\varphi$ is an isomorphism between $\Gamma \rpc (\Gamma \cap U)$ and $G$.
\end{prop}

\begin{ex}
Elder and Willis \cite{eld-wil} considered the \sch completion $G_{m,n}$ of the Baumslag-Solitar group $\mathrm{BS}(m,n) = \left\langle t,x \mid tx^{m}t^{-1} = x^n\right\rangle$ with respect to the commensurated subgroup $\left\langle x\right\rangle$. Theorem \ref{thm-sch} can be applied and yields that $G_{m,n}$ is compactly presented. However in this case this can be seen more directly because $G_{m,n}$ coincides with the closure of $\mathrm{BS}(m,n)$ in the automorphism group of its Bass-Serre tree, and therefore $G_{m,n}$ acts on a locally finite tree with compact vertex stabilizers. It follows that $G_{m,n}$ is Gromov-hyperbolic, and consequently automatically compactly presented.
\end{ex}

The next example shows that the almost automorphism group $\Daaut$ is a \sch completion of the Higman-Thompson group $V_{d,k}$. In particular Neretin's group is the \sch completion of $V_{d,2}$ with respect to an infinite locally finite subgroup, a point of view which does not seem to appear in the literature. This will be generalized in Theorem \ref{aaut-sch}.

\begin{ex}
We use the notation introduced at the beginning of Section \ref{sec-prel}. Let us consider the t.d.l.c.\ group $\Daaut$ and its compact open subgroup $W_k(D)$, which is easily seen not to contain any non-trivial normal subgroup of $\Daaut$. The Higman-Thompson group $V_{d,k}$ is a dense subgroup intersecting $W_k(D)$ along $D_{\infty}^k$ by Lemma \ref{lem-inter}. So it follows from Proposition \ref{lem-sw} that the group $\Daaut$ is isomorphic to the \sch completion $V_{d,k} \rpc D_{\infty}^k$.

However, note that compact presentability of $\Daaut$ cannot be obtained by applying Theorem \ref{thm-sch} because $D_{\infty}^k$ is not finitely generated.
\end{ex}

\subsection{Presentation of $\Gamma \rpc \Lambda$}

We will now prove the main result of this section, namely Theorem \ref{thm-sch}, which will follow from Proposition \ref{prop2-rpc}.

From now $\Gamma$ is a finitely presented group and $\Lambda$ a finitely generated commensurated subgroup. Recall that by abuse of notation, we still denote by $\Gamma$ and $\Lambda$ their images in the group $\Gamma \rpc \Lambda$. We let $S = \left\{s_1, \ldots, s_n, \ldots, s_m\right\}$ be a finite generating set of $\Gamma$ such that the elements $s_1, \ldots, s_n$ generate $\Lambda$. It follows from Lemma \ref{decomp-rpc} that $S \cup \overline{\Lambda}$ is a compact generating set of $\Gamma \rpc \Lambda$. To prove that the group $\Gamma \rpc \Lambda$ is compactly presented, we will consider a set $R = R_1 \cup R_2 \cup R_3 \cup R_4$ of relations of bounded length in $\Gamma \rpc \Lambda$, and prove that it is a set of defining relations, i.e.\ that relations of $R$ generate all the relations in $\Gamma \rpc \Lambda$.

We let $R_1$ be a set of words so that $\left\langle S \mid R_{1} \right\rangle$ is a finite presentation of $\Gamma$.

Let us also consider relations corresponding to the inclusion $\Lambda \leq \overline{\Lambda}$ in the group $\Gamma \rpc \Lambda$. That is, for every $i \in \left\{1, \ldots, n\right\}$, we let $\bar{s_i} \in \overline{\Lambda}$ be such that $s_i = \bar{s_i}$ in $\Gamma \rpc \Lambda$, and denote by $R_{2}$ the set of words $s_i \bar{s_i}^{-1}$.

We denote by $R_{3}$ the set of relations of the form $u_1u_2=u_3$, $u_i \in \overline{\Lambda}$.

Now let us define the abstract group $G_1 = \left\langle S \cup \overline{\Lambda} \mid R_1, R_2, R_3 \right\rangle$. Note that by construction there is a homomorphism $G_1 \rightarrow \Gamma \rpc \Lambda$.

\begin{prop} \label{prop1-rpc}
Let $w$ be a word in the elements of $S$ and $u \in \overline{\Lambda}$. If the word $u^{-1}w$ represents the identity in $\Gamma \rpc \Lambda$, then it already represents the identity in $G_1$.
\end{prop}

\begin{proof}
The fact that $u^{-1}w$ represents the identity in $\Gamma \rpc \Lambda$ means that the element represented by $w$ lies in $\Gamma \cap \overline{\Lambda}$, which is reduced to $\Lambda$ according to Lemma \ref{inter-rpc}. Therefore there exists a word $w_{\Lambda}$ in the letters $s_1, \ldots, s_n$, so that $w = w_{\Lambda}$ in $\Gamma \rpc \Lambda$. But thanks to $R_1$, the relation $w = w_{\Lambda}$ is also satisfied in $G_1$. Now for each letter of $w_{\Lambda}$ we can apply a relation from $R_2$ to obtain a word $w_{\overline{\Lambda}}$ in the letters $\bar{s_1}, \ldots, \bar{s_n}$, so that $w = w_{\overline{\Lambda}}$ in $G_1$. Consequently the relation $w_{\overline{\Lambda}} = u$ holds in $\Gamma \rpc \Lambda$, and thanks to $R_3$ this relation also holds in $G_1$, meaning that $w=u$ in $G_1$.
\end{proof}

We finally consider a last family of relations in $\Gamma \rpc \Lambda$. According to Lemma \ref{decomp-rpc}, for every $i \in \left\{1, \ldots, m\right\}$ and $u \in \overline{\Lambda}$, we can pick some $u' \in\overline{\Lambda}$ and some word $w \in S^*$ so that $s_i u = u' w$ in $\Gamma \rpc \Lambda$. We denote by $R_4$ the set of corresponding relations.

Now let us define the abstract group $G_2 = \left\langle S \cup \overline{\Lambda} \mid R_1, R_2, R_3, R_4 \right\rangle$. Note that the group $G_2$ is a quotient of $G_1$.

\begin{prop} \label{prop2-rpc}
The natural homomorphism $G_2 \rightarrow \Gamma \rpc \Lambda$ is an isomorphism. 
\end{prop}

\begin{proof}
It is clear that this morphism is onto because $S \cup \overline{\Lambda}$ is a generating set of $\Gamma \rpc \Lambda$. So we only have to prove that it is injective. For this, let us consider a word $w$ in the elements of $S \cup \overline{\Lambda}$ representing the identity in $\Gamma \rpc \Lambda$. We want to prove that $w$ represents the identity in $G_2$. Applying successively relators from $R_4$, we can move each occurrence of an element of $\overline{\Lambda}$ in $w$ to the left, and obtain a word $w'$ of the form $w' = u_1 \cdots u_k s_{i_1} \cdots s_{i_{\ell}}$, with $u_i \in \overline{\Lambda}$, $s_j \in S$, so that $w = w'$ in $G_2$. Now thanks to $R_3$, the word $w'$ can be transformed into a word $w''$ of the form $w'' = u s_{i_1} \cdots s_{i_{\ell}}$, $u \in \overline{\Lambda}$. But now since $w$ represents the identity in $\Gamma \rpc \Lambda$, the same holds for $w''$. Therefore by Proposition \ref{prop1-rpc}, the word $w''$ represents the identity in $G_1$, and a fortiori it also represents the identity in $G_2$, the latter being a quotient of $G_1$. It follows that the word $w$ represents the identity in $G_2$, and the proof is complete.
\end{proof}

\begin{rmq}
Here we do not try to get an estimate on the Dehn function of $\Gamma \rpc \Lambda$, because a careful reading of the proof reveals that the best we could hope in this level of generality is to obtain that the Dehn function of $\Gamma \rpc \Lambda$ is bounded by the Dehn function of $\Gamma$. However this would be far from being sharp, as for example the Baumslag-Solitar group $\mathrm{BS}(1,n)$ has an exponential Dehn function for $n \geq 2$ (see for instance \cite{Gro-Herm}), whereas its \sch completion $\mathbb{Q}_n \rtimes_n \mathbb{Z}$ is Gromov-hyperbolic, and therefore has a linear Dehn function.
\end{rmq}

\section{Almost automorphism groups associated with regular branch groups} \label{Gaaut-type} 

Almost automorphisms of $\treed$ are homeomorphisms of the boundary $\bordd$ which locally coincide with a tree automorphism. It seems natural to extend this definition by considering a subgroup $G \leq \autd$, and homeomorphisms of $\bordd$ which locally coincide with an element of $G$. In other words, we want to define homeomorphisms of $\bordd$ which are piecewise in $G$. It turns out that the notion naturally appearing for $G$ is self-similarity, whose definition is recalled below. 

Note that we restrict ourselves to the case $k=d$ for simplicity, but the results could naturally be extended to almost automorphism subgroups of $\aaut$.

\subsection{Preliminary material}

This section is devoted to reviewing basic definitions and facts about self-similar and branch groups, and establishing some preliminary results. We refer the reader to \cite{Nek-ssg,branch} for more on self-similar and branch groups.

Recall that vertices of $\treed$ are labeled by words over a finite alphabet $X$ of cardinality $d$, and we freely identify a vertex with the word associated to it.

If $G$ is a subgroup of the automorphism group $\autd$ and if $n \geq 0$, we will denote by $G_n$ the $n$th level stabilizer of $G$, that is the subgroup of $G$ fixing pointwise the $n$th level of $\treed$. Note that $G_n$ is always a finite index subgroup of $G$, but the converse is far from true because there may exist some finite index subgroup of $G$ not containing any level stabilizer. This motivates the following definition.

\begin{defi}
A finite index subgroup of $G$ is a congruence subgroup if it contains some level stabilizer.
\end{defi}

If $g \in \autd$ is an automorphism and $v \in X^*$ is a vertex of $\treed$, the section of $g$ at $v$ is the unique automorphism $g_v$ of $\treed$ defined by the formula \[ g(vw) = g(v)g_v(w) \] for every $w \in X^*$.

\begin{defi}
A subgroup $G \leq \autd$ is self-similar if every section of every element of $G$ is an element of $G$.
\end{defi}

Self-similar groups appear naturally when studying holomorphic dynamics and fractal geometry. The study of self-similar groups is also motivated by the fact that this class contains examples of groups exhibiting some exotic behavior. Among self-similar groups is a class of groups which is better understood, namely contracting self-similar groups.

\begin{defi}
A self-similar group $G$ is said to be contracting if there exists a finite subset $\mathcal{N} \leq G$ such that for every $g \in G$, there exists $k \geq 1$ so that all the sections of $g$ of level at least $k$ belong to $\mathcal{N}$.
\end{defi}

Here we are interested in a particular class of self-similar groups, namely regular branch groups, whose definition is recalled below.

\begin{defi}
Let $G \leq \autd$ be a self-similar group. By definition, $G$ comes equipped with an injective homomorphism $\psi: G \rightarrow G \wr \mathrm{Sym}(d)$ (sometimes called the wreath recursion). We say that $G$ is regular branch over its finite index subgroup $K$ if $\psi(K)$ contains $K \times \ldots \times K$ as a subgroup of finite index.
\end{defi}

\begin{rmq}
We note that being regular branch is stable by taking the topological closure in $\autd$. More precisely, if $G$ is regular branch over $K$ then the closure of $G$ is regular branch over the closure of $K$. Note also that if $K$ contains some level stabilizer of $G$ then its closure contains the stabilizer of the same level in the closure of $G$, so being regular branch over a congruence subgroup is also stable by taking the topological closure.
\end{rmq}

The most popular example of a self-similar group is the Grigorchuk group of intermediate growth introduced in \cite{grig80}. It is a regular branch group, branching over a subgroup containing its stabilizer of level $3$. Other examples are the Gupta-Sidki group as well as the Fabrykowski-Gupta group, which are regular branch over their commutator subgroup, the latter containing their level $2$ stabilizer. For the definitions and properties of these groups we refer the reader to Sections $6$ and $8$ of \cite{BG-Hecke}. In view of Theorem \ref{thm-cp-gaaut}, we note that all these examples are contracting.

%\begin{ex}
%The group $\widetilde{G}$ introduced in \cite{BG-Hecke} is regular branch over a subgroup containing its stabilizer of level $4$ (see Proposition 5.9 of \cite{BG-Hecke}). 
%\end{ex}

In the following standard lemma, a proof of which can be consulted in \cite[Lemma 10]{ZS-haus}, the isomorphism is obtained via the wreath recursion, which is usually omitted.

\begin{lem} \label{lem-regbr}
Let $H \leq \autd$ be a regular branch group, branching over a subgroup containing the level stabilizer $H_{s}$. Then for every $n \geq s$, the level stabilizer $H_{n+1}$ is isomorphic to $H_n \times \ldots \times H_n$.
\end{lem}

If $H$ is a subgroup of the automorphism group $\autd$, it is in general very hard to describe its topological closure in $\autd$. In the case of the Grigorchuk group, the closure has been described by Grigorchuk in \cite{grig-solved}. We will use a generalization of this result due to Sunic, which is the following:

\begin{prop} \label{prop-zs}
Let $H \leq \autd$ be a regular branch group, branching over a subgroup containing the level stabilizer $H_{s}$, and let $G$ be the topological closure of $H$ in $\autd$. Then an element $\gamma \in \autd$ belongs to $G$ if and only if for every section $\gamma_{v}$ of $\gamma$, there exists an element of $H \leq \autd$ acting like $\gamma_v$ up to and including level $s+1$.
\end{prop}

\begin{proof}
The statement is a reformulation of the implication $(ii) \Rightarrow (i)$ of Theorem $3$ of \cite{ZS-haus}. Note that the author requires level transitivity in the definition of a regular branch group, but the proof given there does not use this assumption. 
\end{proof}

This description of the closure of a regular branch group allows us to deduce the following result, which does not seem to appear in the literature, and which may be of independent interest. 
 
\begin{prop} \label{prop-intclos}
Let $H \leq \autd$ be a regular branch group, branching over a congruence subgroup, and let $G$ be the topological closure of $H$ in $\autd$. Then the intersection in $\autd$ between $G$ and $H \wr \mathrm{Sym}(d)$ is equal to $H$.
\end{prop}

\begin{proof}
The inclusion $H \subset G \cap (H \wr \mathrm{Sym}(d))$ is clear because $H \subset G$ is always true and $H \subset  H \wr \mathrm{Sym}(d)$ is satisfied by self-similarity. To prove that equality holds, we prove that $H$ and $G \cap (H \wr \mathrm{Sym}(d))$ have the same index in the group $H \wr \mathrm{Sym}(d)$. 

Assume that $H$ is branching over a subgroup containing $H_{s}$. By multiplicativity of the index, we have \[ \left[ H \wr \mathrm{Sym}(d) : H_{s+1} \right] = \left[ H \wr \mathrm{Sym}(d) : H \right] \times \left[ H : H_{s+1} \right], \] that is \[ \left[ H \wr \mathrm{Sym}(d) : H \right] = \frac{\left[ H \wr \mathrm{Sym}(d) : H_{s+1} \right]}{\left[ H : H_{s+1} \right]}.\]
Now the number of possibilities for the action of an element of $H \wr \mathrm{Sym}(d)$ on the first level is $|\mathrm{Sym}(d)| = d!$. Moreover the first level stabilizer of $H \wr \mathrm{Sym}(d)$ is $H \times \ldots \times H$, so \begin{equation} \label{eq-ind} \left[ H \wr \mathrm{Sym}(d) : H_{s+1} \right] = d! \left[ H \times \ldots \times H : H_{s+1} \right]. \end{equation} Furthermore we can apply Lemma \ref{lem-regbr} to obtain that $H_{s+1}$ is equal to $H_{s} \times \ldots \times H_{s}$, which yields \[ \left[ H \times \ldots \times H : H_{s+1} \right] = \left[ H \times \ldots \times H : H_{s} \times \ldots \times H_{s} \right] = \left[ H : H_s \right]^d. \] Going back to (\ref{eq-ind}), we obtain \[ \left[ H \wr \mathrm{Sym}(d) : H \right] = \frac{d! \left[ H : H_s \right]^d}{\left[ H : H_{s+1} \right]}.\]

Let us now compute the index of $G \cap (H \wr \mathrm{Sym}(d))$ in $H \wr \mathrm{Sym}(d)$. According to Proposition \ref{prop-zs}, an element $\gamma \in \autd$ belongs to $G$ if and only if for every section $\gamma_{v}$ of $\gamma$, there exists an element of $H$ acting like $\gamma_v$ up to level $s+1$. Since elements of $H \wr \mathrm{Sym}(d)$ have all their sections of level at least $1$ in $H$, it follows that an element $\gamma \in H \wr \mathrm{Sym}(d)$ belongs to $G$ if and only if there exists an element of $H$ acting like $\gamma$ up to level $s+1$. It follows that the index of $G \cap (H \wr \mathrm{Sym}(d))$ in $H \wr \mathrm{Sym}(d)$ is the number of possibilities for the action on level $s+1$ for $H \wr \mathrm{Sym}(d)$, divided by the number of possibilities for the action on level $s+1$ for $H$. The latter being $\left[H : H_{s+1}\right]$ and the former being $d! \left[H : H_s \right]^d$, we have \[ \left[ H \wr \mathrm{Sym}(d) : G \cap (H \wr \mathrm{Sym}(d)) \right] = \frac{d! \left[ H : H_s \right]^d}{\left[ H : H_{s+1} \right]}.\]
\end{proof}

\subsection{Definition of the groups}

Let $G \leq \autd$ be a self-similar group. We will say that an almost automorphism of $\treed$ is piecewise of type $G$ if it can be represented by a triple $(\psi,T,T')$ such that $T,T'$ are finite rooted complete subtrees of $\treed$ and $\psi: \treed \setminus T \rightarrow  \treed \setminus T'$ belongs to $G$ on each connected component, after the natural identification of each connected component of $\treed \setminus T$ and $\treed \setminus T'$ with $\treed$. We observe that by self-similarity, if a triple $(\psi_1,T_1,T_1')$ is such that $\psi_1: \treed \setminus T_1 \rightarrow  \treed \setminus T_1'$ belongs to $G$ on each connected component, then for any equivalent triple $(\psi_2,T_2,T_2')$ such that $T_2$ (resp.\ $T_2'$) contains $T_1$ (resp.\ $T_1'$), then $\psi_2: \treed \setminus T_2 \rightarrow  \treed \setminus T_2'$ belongs to $G$ on each connected component. It follows from this observation that the set of almost automorphisms which are piecewise of type $G$ is a subgroup of $\aautd$, which will be denoted by $\Gaaut$. Note that $\Gaaut$ obviously contains the group $G$.

It is worth pointing out that the definition of the group requires neither topological (e.g.\ closed) nor finiteness (e.g.\ finitely generated) assumption on $G$.

Following \cite{Nek-fp}, we let $L(G) \leq \autd$ be the embedded copy of $G$ acting on the subtree hanging below the first vertex of level $1$, and being the identity elsewhere. Since the Higman-Thompson group $V_d$ acts transitively on the set of proper balls of $\bordd$, it is not hard to see that the group $\Gaaut$ is generated by $V_d$ together with $L(G)$. See Lemma 5.12 in \cite{Nek-fp} for details. In particular if $G$ is a finitely generated self-similar group, then $\Gaaut$ is finitely generated as well. 

The first example of such a group was considered by Röver when $G$ is the first Grigorchuk group. He proved that $\Gaaut$ is finitely presented and simple \cite{Rov-cfpsg}. Then Nekrashevych \cite{Nek-cp} introduced the group $\Gaaut$ for an arbitrary self-similar group $G$ and generalized both simplicity and finiteness results (see Theorem 4.7 in \cite{Nek-fp} and Theorem \ref{thm-nekfp} cited below).

\begin{rmq} \label{rmq-cor-normal}
It is worth noting that $\Gaaut$ is always a dense subgroup of $\aautd$, since it contains the subgroup $V_d$ which is already dense. In particular if $N$ is a non-trivial normal subgroup of $\Gaaut$, then the closure of $N$ in $\aautd$ is normalized by the closure of $\Gaaut$, which is $\aautd$. By simplicity of the latter, the closure of $N$ has to be equal to $\aautd$. This proves that any non-trivial normal subgroup of $\Gaaut$ is dense in $\aautd$. In particular $G$ can not contain any non-trivial normal subgroup of $\Gaaut$.
\end{rmq}

%The following result is a modified version of Lemma 4.1 from \cite{Nek-fp} (see also references given there). It means that non-trivial normal subgroups of $\Gaaut$ are in a sense very large. Its proof is just an adaptation of the aforementioned lemma, so we do not repeat it here. 

%\begin{lem}
%Let $G \leq \autd$ be a self-similar group and let $N$ be a non-trivial normal subgroup of $\Gaaut$. Then there exists a ball $B$ in $\bordd$ such that $N$ contains the normal closure of the commutator subgroup of the group consisting of elements of $\Gaaut$ acting trivially outside $B$.
%\end{lem}

%We easily deduce the following corollary, which will be used in the proof of Theorem \ref{aaut-sch}.

%\begin{cor} \label{cor-normal}
%For any self-similar group $G \leq \autd$, the normal core of $G$ in $\Gaaut$ is trivial.
%\end{cor}

\subsection{Almost automorphism groups arising as \sch completions}

The main result of this section is the following.

\begin{thm} \label{aaut-sch}
Let $H \leq \autd$ be a regular branch group, branching over a congruence subgroup, and let $G$ be the topological closure of $H$ in $\autd$. Then the inclusion of $\Haaut$ in $\Gaaut$ induces an isomorphism of topological groups between $\Haaut \rpc H$ and $\Gaaut$.
\end{thm}

For example this brings a new perspective to the topologically simple group constructed in \cite{BEW}: this is the \sch completion of Röver's group \cite{Rov-cfpsg} with respect to the Grigorchuk group.

Theorem \ref{aaut-sch} will be proved at the end of this section. We begin by showing how to endow the group $\Gaaut$ with a natural topology when $G$ is a closed regular branch group. We will need the following:

\begin{prop}
Any regular branch group $G \leq \autd$ is commensurated in $\Gaaut$.
\end{prop}

\begin{proof}
Since $\Gaaut$ is generated by $V_d$ and $L(G)$, it is enough to prove that these two subgroups commensurate $G$.

Let us first prove that $V_d$ commensurates $G$. Henceforth we assume that $K$ is a subgroup of $G$ over which $G$ is branching. For every finite rooted complete subtree $T$ of $\treed$, we denote by $K_T$ the subgroup of $\autd$ fixing pointwise $T$ and acting by an element of $K$ on each subtree hanging below a leaf of $T$. Since $G$ is regular branch over $K$, $K_T$ is a finite index subgroup of $G$ for every finite rooted complete subtree $T$. Now if $\sigma \in V_d$ and if $T,T'$ are respectively the domain and range tree of the canonical representative triple of $\sigma$, we easily check that $\sigma K_T \sigma^{-1} = K_{T'}$. So conjugation by $\sigma$ sends a finite index subgroup of $G$ to another finite index subgroup of $G$, which exactly means that $\sigma$ commensurates $G$.

Now let us prove that $L(G)$ commensurates $G$. It is classic that since $K$ is a finite index subgroup of $G$, there exists a finite index subgroup $N$ of $K$ that is normal in $G$. Therefore $\psi(G)$ contains $N \times \ldots \times N$ as a finite index subgroup, and the latter is normalized by $L(G)$ because $N$ is normal in $G$. This proves an even stronger result than commensuration, namely the existence of a finite index subgroup of $G$ which is normalized by $L(G)$.
\end{proof}

Now assume that $G \leq \autd$ is a closed regular branch group. Examples of such groups include the topological closure of any of the finitely generated regular branch groups mentioned earlier. In this context, the group $G$ comes equipped with a profinite topology, inherited from the profinite topology of $\autd$. The fact that $G$ is commensurated in $\Gaaut$ together with Lemma \ref{bourb} allows us to extend the topology of $G$ to the larger group $\Gaaut$:

\begin{prop}
Assume that $G \leq \autd$ is a closed regular branch group. Then there exists a (unique) group topology on $\Gaaut$ turning $G$ into a compact open subgroup. In particular $\Gaaut$ is a t.d.l.c.\ compactly generated group.
\end{prop}

We now prove some preliminary results which will be used in the proof of Theorem \ref{aaut-sch}. \newline

\emph{Until the end of this section, $H \leq \autd$ is a regular branch group, branching over a congruence subgroup, and $G$ is the topological closure of $H$ in $\autd$.}

\begin{prop} \label{prop-dense-subgroup}
$\Haaut$ is a dense subgroup of $\Gaaut$.
\end{prop}

\begin{proof}
We let $L$ be a congruence subgroup of $G$ over which $G$ is branching, and we denote by $K$ the closure of $L$ in $\autd$. For every finite rooted complete subtree $T$ of $\treed$, we still denote by $K_T$ the subgroup of $\autd$ fixing pointwise $T$ and acting by an element of $K$ on each subtree hanging below a leaf of $T$. Note that since $L$ contains some level stabilizer of $H$, the subgroup $K$ contains some level stabilizer of $G$ and is therefore an open subgroup of $G$. It follows that $(K_T)$ forms a basis of neighbourhoods of the identity in $G$, when $T$ ranges over all finite rooted complete subtrees. By definition of the topology, it is also a basis of neighbourhoods of the identity in $\Gaaut$.

Let $g$ be an element of $G$. By definition there exists a sequence $(h_n)$ of elements of $H$ converging to $g$. Since $L$ has finite index in $H$, we may assume that all the elements $h_n$ lie in the same left coset of $L$, that is, that there exists $h \in H$ such that $h_n \in h L$ for every $n$. From this we deduce that $g \in h K$.

Now let $\gamma$ be an element of $\Gaaut$. We will prove that $\Haaut$ intersects every neighbourhood of $\gamma$. Let $(\psi,T,T')$ be a triple representing $\gamma$ such that $\psi: \treed \setminus T \rightarrow  \treed \setminus T'$ belongs to $G$ on each connected component of $\treed \setminus T$. This means that for every leaf $v$ of $T$, there exists an element $g_v \in G$ so that $\psi$ sends the subtree hanging below the leaf $v$ to a subtree hanging below a leaf of $T'$ via the element $g_v$. According to the above remark, there exists some element $h_v \in H$ such that $h_v^{-1}g_v \in K$. Now let us consider the almost automorphism $\hat{\gamma}$ represented by the triple $(\hat{\psi},T,T')$, where $\hat{\psi}$ induces the same bijection between the leaves of $T$ and the leaves of $T'$, but does not act on the subtree hanging below the leaf $v$ by $g_v$ but by the element $h_v$. By construction, we have $\hat{\gamma} \in \Haaut$ and $\hat{\gamma}^{-1} \gamma \in K_T$. Since on the one hand we can choose $T$ to be as large as we want, and on the other hand $(K_T)$ is a basis of neighbourhoods of the identity, we obtain that $\Haaut$ intersects every neighbourhood of $\gamma$.
\end{proof}

%Having Proposition \ref{lem-sw} in mind naturally leads to the study of the intersection between $\overline{G}$ and $\Gaaut$ in $\mathrm{AAut}_{\overline{G}}(\mathcal{T}_{d})$. 

\begin{prop} \label{prop-G-inter}
The intersection in $\Gaaut$ between $G$ and $\Haaut$ is equal to $H$.
\end{prop}

\begin{proof}
The inclusion $H \subset G \cap \Haaut$ being clear, we only have to prove the reverse inclusion. 
First note that the intersection between $\autd$ and $\Haaut$ is the increasing union for $n \geq 0$ of the subgroups $H \wr \mathrm{Aut}_n$, where $\mathrm{Aut}_n$ is the subgroup of $\autd$ consisting of elements whose sections of level $n$ are trivial; and the permutational wreath product is associated to the action of $\mathrm{Aut}_n$ on the vertices of level $n$. In particular \[ G \cap \Haaut = \bigcup_{n \geq 0} G \cap (H \wr \mathrm{Aut}_n). \] Let us prove by induction on $n \geq 0$ that $G \cap (H \wr \mathrm{Aut}_n)$ is reduced to $H$. This is true for $n=0$ by definition, and true for $n=1$ according to Proposition \ref{prop-intclos}. Assume that this is true for some $n \geq 1$, and let $\gamma \in G \cap (H \wr \mathrm{Aut}_{n+1})$. Then every section of level $1$ of $\gamma$ lies in $H \cap (H \wr \mathrm{Aut}_n)$, which is reduced to $H$ by induction hypotheses. Therefore $\gamma \in G \cap (H \wr \mathrm{Sym}(d))$, which is also equal to $H$ by Proposition \ref{prop-intclos}. So we have proved the induction step, namely $G \cap (H \wr \mathrm{Aut}_{n+1}) = G$, and consequently $G \cap \Haaut = H$.
\end{proof}

We are now ready to prove the main result of this section.

\begin{proof}[Proof of Theorem \ref{aaut-sch}]
The group $\Gaaut$ admits $\Haaut$ as a dense subgroup by Proposition \ref{prop-dense-subgroup}, and the latter intersects the compact open subgroup $G$ along $H$ according to Proposition \ref{prop-G-inter}. Moreover Remark \ref{rmq-cor-normal} prevents $G$ from containing any non-trivial normal subgroup of $\Gaaut$, so the conclusion follows from Proposition \ref{lem-sw}.
\end{proof}

\subsection{Proof of Theorem \ref{thm-cp-gaaut}}

We conclude by proving Theorem \ref{thm-cp-gaaut}. The only missing argument is a recent result of Nekrashevych, generalizing the previous example of R\"{o}ver \cite{Rov-cfpsg}.

\begin{thm} [\cite{Nek-fp}, Theorem 5.9] \label{thm-nekfp} \enskip \newline
If $H \leq \autd$ is a finitely generated, contracting self-similar group, then $\Haaut$ is finitely presented.
\end{thm}

\begin{proof}[Proof of Theorem \ref{thm-cp-gaaut}]
Let $H$ be a finitely generated, contracting regular branch group, branching over a congruence subgroup, having $G$ for topological closure in $\autd$. Then by Theorem \ref{aaut-sch} $\Gaaut$ is isomorphic to the \sch completion $\Haaut \rpc H$. Now according to Theorem \ref{thm-nekfp} the group $\Haaut$ is finitely presented, and $H$ is finitely generated by assumption, so the conclusion follows from Theorem \ref{thm-sch}.
\end{proof}

We make a brief comment on the fact that any group $G$ appearing in Theorem \ref{thm-cp-gaaut} can be explicitly described in terms of the group of which it is the topological closure. Indeed, if $H$ is a finitely generated, contracting regular branch group, branching over a subgroup containing $H_s$, having $G$ for topological closure in $\autd$; then Proposition \ref{prop-zs} yields that elements of $G$ are exactly the automorphisms having all their sections acting like an element of $H$ up to level $s+1$. One can rephrase this in terms of patterns and finitely constrained groups (see \cite{ZS-haus}), by saying that $G$ is the finitely constrained group defined by allowing all the patterns of size $s+1$ appearing in $H$.

\nocite{*}
\bibliographystyle{amsalpha}
\bibliography{aaut}

\end{document}